\documentclass{article}

\usepackage{arxiv}

\usepackage[utf8]{inputenc}
\usepackage[T1]{fontenc}
\usepackage{amsmath,amssymb,amsthm,mathtools}
\usepackage{graphicx}
\usepackage[caption=false]{subfig}
\usepackage{multirow}
\usepackage{booktabs}
\usepackage{microtype}
\usepackage{url}
\usepackage{doi}
\usepackage[numbers,sort&compress]{natbib}
\usepackage[ruled,linesnumbered]{algorithm2e}
\usepackage{xcolor}
\usepackage{authblk}
\usepackage{hyperref}

\graphicspath{{figures/}}

\newtheorem{remark}{Remark}
\newtheorem{example}{Example}
\newtheorem{lemma}{Lemma}
\newtheorem{theorem}{Theorem}
\newtheorem{corollary}{Corollary}

\newenvironment{highlights}
  {\par\medskip\noindent\textbf{Highlights}\begin{itemize}}
  {\end{itemize}\medskip}

\title{A Matrix-free Augmented High Order Compact Solver for Variable-Coefficient Biharmonic Problems}

\author[1]{Jin Li}
\author[2]{Kejia Pan\thanks{Corresponding author. Email: \texttt{kejiapan@csu.edu.cn}}}
\author[1]{Xu Qian}
\author[3]{Li-Lian Wang}
\affil[1]{College of Science, National University of Defense Technology, Changsha 410073, China}
\affil[2]{School of Mathematics and Statistics, HNP-LAMA, Central South University, Changsha 410083, China}
\affil[3]{Division of Mathematical Sciences, School of Physical and Mathematical Sciences, Nanyang Technological University, Singapore 637371, Singapore}
\date{}

\renewcommand{\shorttitle}{Matrix-free Augmented HOC Solver for Biharmonic Problems}

\hypersetup{
  colorlinks=true,
  linkcolor=black,
  citecolor=black,
  urlcolor=blue,
  pdftitle={A Matrix-free Augmented High Order Compact Solver for Variable-Coefficient Biharmonic Problems},
  pdfauthor={Jin Li, Kejia Pan, Xu Qian, Li-Lian Wang},
  pdfkeywords={Augmented method; biharmonic equation; clamped boundary conditions; variable coefficients; high-wavenumber problems; compact finite differences}
}

\begin{document}
\maketitle

\begin{abstract}
We propose an augmented high-order compact finite difference method for biharmonic equations with clamped boundary conditions and variable coefficients. Standard mixed-type formulations introduce   
an auxiliary variable, but  its boundary values are unavailable, leaving the resulting discrete systems globally coupled and difficult to solve at large scales. Our key contribution is the development of a new augmented formulation that treats these unavailable boundary values as additional unknowns, reduces the global coupling to a lower-dimensional Schur complement system, and yields decoupled second-order subproblems. The Schur complement is solved by matrix-free GMRES, while the subproblems are handled by FFT-based fast solvers. The method achieves   fourth-order accuracy using compact stencils, and has  $O(n\log n)$ computational complexity, enabling the solution of the biharmonic equation with 
$1024^3$ degrees of freedom within several minutes. To the best of our knowledge, this level of computational efficiency has not previously been achieved in either the literature or practice. Using energy estimates and Fourier analysis, we derive a new $L^2$-estimate for Poisson equations with inexact Dirichlet boundary and then prove the convergence of the proposed scheme. We provide ample
numerical experiments to confirm the accuracy, efficiency, and further apply the fast and accurate solver to triharmonic equations, high-wavenumber problems, Stokes flow, and plate bending problems.
\end{abstract}



\begin{highlights}
\item A genuinely decoupled HOC difference solver is developed for variable-coefficient biharmonic problems in 2D and 3D.
\item Fourth-order accuracy is retained for variable coefficients and higher-order equations.
\item The matrix-free algorithm exhibits quasi-linear $O(n\log n)$ complexity.
\item Three dimensional grids up to $1024^3$ are solved within minutes.
\end{highlights}

\section{Introduction}
The biharmonic equation is a classical fourth-order elliptic model arising in many areas of applied mathematics and mechanics, including elasticity, plate theories, and streamfunction formulations of viscous incompressible flows. 
Compared with standard second-order elliptic problems, its numerical treatment is considerably more difficult because of the higher-order operator and the associated boundary conditions. 
In particular, it remains challenging to construct discretizations that simultaneously achieve high-order accuracy, compactness, and computational efficiency. 
This difficulty becomes even more pronounced for problems with clamped boundary conditions and in large-scale computations.

Numerical methods for the biharmonic equation can be broadly classified into two categories. 
The first category consists of direct discretizations of the fourth-order operator. 
The second category is based on mixed reformulations, in which the original fourth-order equation is reduced to a coupled system of second-order equations. 
The present work belongs to the second category.
% Our goal is to develop a numerical scheme, based on mixed formulations, that achieves high order accuracy, compactness, and high computational efficiency.

The first category is to design a direct discretization of the fourth order operator using finite difference or finite element methods. The most common direct finite difference schemes are second-order accurate and noncompact~\cite{DM-FD1967,DM-FD1983,muller2020optimal}. They are typically obtained by applying the discrete Laplace operator twice, which leads to wide stencils, such as a 13-point stencil in 2D and a 25-point stencil in 3D. This also results in an ill-conditioned coefficient matrix whose condition number is of order $O(h^{-4})$. In addition, such noncompact discretizations usually require special modifications near the boundary in order to maintain consistency and accuracy. Later, combined compact finite difference schemes were introduced, in which the solution and its first-order derivatives are treated as primary unknowns~\cite{DM-FD2002,ben2025num,ben2009compact}. These methods improve compactness, but the number of unknowns increases with the spatial dimension.
In the finite element literature, direct discretizations of fourth-order problems started from the classical $C^1$ conforming elements, such as the Argyris~\cite{DM-FE1968}, Bell~\cite{DM-FE1969}, and Bogner--Fox--Schmit elements~\cite{DM-FE1965}, and were later extended to nonconforming elements such as the Morley element~\cite{DM-FEM1971}. However, much like direct finite difference methods, these direct finite element approaches still inherit the intrinsic difficulties of fourth-order problems.
% In short, direct discretizations make it difficult to achieve high-order accuracy, compactness, and computational efficiency at the same time.

The second category is based on mixed formulations of the biharmonic equation, which effectively avoid the difficulties of directly discretizing fourth-order operators.
By introducing auxiliary variables, the original fourth-order problem is reduced to a coupled system of second-order equations.
In principle, this makes the subproblems closer to standard Poisson- or Helmholtz-type equations, for which numerical discretizations and fast solvers are much more mature. 
In practice, however, this advantage can only be fully exploited when genuine decoupling is achieved.
Existing work has paid relatively little attention to this issue. 
This is the main motivation for the present work.

The works, based on finite difference discretization, can be traced back to the coupled-equation approach of Smith and Ehrlich \cite{Smith1970,Ehrlich1971}, where the biharmonic equation was reduced to two second-order problems and an additional iteration was used to recover the auxiliary boundary conditions. 
However, the resulting subproblems remained nontrivially coupled, and the convergence behavior depended strongly on boundary treatment and iteration parameters. 
Later, Stephenson \cite{Stephenson1984} proposed compact mixed finite difference schemes by introducing first-order derivatives as additional unknowns, and related developments yielded more compact and more accurate discretizations than standard formulations~\cite{AtlasDymGuptaManohar1998}. 
Nevertheless, the number of unknowns increased with the spatial dimension, and the resulting algebraic systems became correspondingly larger. 
More recently, mixed high order compact finite difference methods have been developed, with particular emphasis on high order accuracy and compact boundary discretization~\cite{pan2025fourth}.
Even so, the coupled discrete structure still limits the overall computational efficiency, especially for large-scale computations.

% A similar situation appears in mixed finite element method. 
% This methods avoid the construction of $C^1$ finite element spaces by introducing auxiliary variables and reformulating the biharmonic equation as a lower-order coupled system. 
% Representative examples include the C-R method~\cite{CR1974}, the H-M and H-J methods~\cite{Herrmann1967,Miyoshi1972,Johnson1973}, and more recent formulations~\cite{ArnoldHu2021,ChenHuang2020,HuMaZhang2021}. 
% Although these developments have greatly advanced mixed finite element methods for biharmonic problems, the construction of suitable stress spaces remains highly nontrivial, especially when simple basis functions, optimal convergence, and a unified treatment on different grids are desired. 

A similar situation appears in mixed finite element methods.
These methods avoid the construction of $C^1$ finite element spaces by
introducing auxiliary variables and reformulating the biharmonic
equation as a lower-order coupled system.
Representative examples include the C-R method~\cite{CR1974},
the H-M and H-J methods~\cite{Herrmann1967,Miyoshi1972,Johnson1973},
and more recent formulations~\cite{ArnoldHu2021,ChenHuang2020,HuMaZhang2021}.
More recently, related high-order hybrid formulations based on the
Hybrid High-Order (HHO) method have adopted a boundary-augmentation
strategy, in which the missing boundary trace of the auxiliary variable
is treated as an additional unknown and recovered through a boundary
iterative problem~\cite{antonietti2024iterative}.
Although these developments have greatly advanced high-order
discretizations for biharmonic problems, the construction of suitable
stress spaces remains highly nontrivial in conforming mixed finite
element approaches, especially when simple basis functions, optimal
convergence, and a unified treatment on different grids are desired.

Although mixed reformulations replace the original fourth-order equation by a coupled system of second-order equations, the associated boundary conditions usually prevent a complete decoupling after discretization.
In particular, the boundary values of the auxiliary variables are not explicitly available and must still be determined together with the primary unknown.
As a result, the discrete system remains coupled, which limits the construction of mixed methods that are simultaneously high-order, compact, and computationally efficient.

% In our view, augmented strategy seems to be an effective way to address this difficulty.
% Augmented strategies have been widely used for problems in which certain boundary or interface quantities are not explicitly prescribed but are still essential for obtaining a well-conditioned and efficiently solvable formulation. 
% By introducing a small number of auxiliary unknowns on boundaries or interfaces, the main part of the problem can often be reduced to standard elliptic subproblems, while the remaining constraints are enforced through a lower-dimensional augmented system. 
% Since this augmented system need not be formed explicitly, its solution can often be obtained at a cost only weakly dependent on the mesh size.
% Such ideas have been successfully applied in incompressible flow\cite{li2007aug,pan2024aug}, interface problems\cite{li2017aug}, irregular domains\cite{pan2021aug}, free boundary problems\cite{li2015aug}, and fluid-structure interaction problems\cite{li2016aug}. 
% These developments suggest that a similar augmented strategy may help overcome the lack of genuine decoupling in mixed formulations of biharmonic equations.

In this work, we propose a matrix-free augmented HOC solver for biharmonic equations with variable coefficients and clamped boundary conditions. The main contributions are as follows:
\vskip 3pt
\begin{itemize}
	\item In standard mixed formulations, the boundary values of the auxiliary variable are not explicitly available, leading to a globally coupled discrete system. 
	We introduce these unknown values as augmented variables, encapsulating the global coupling in a lower-dimensional Schur complement system solely for these variables. 
	Once the augmented variables are computed, the original coupled mixed system reduces to a set of decoupled second-order subproblems.
	\vskip 3pt
	
	\item We develop a matrix-free solution procedure for the augmented system. 
	The Schur complement is not assembled explicitly; instead, its matrix-vector products are computed by solving the corresponding decoupled second-order subproblems from the compact discretization. 
	These subproblems are solved via FFT-based fast solvers, facilitating an efficient implementation of the fourth-order compact scheme with an overall complexity of $O(n\log n)$.
	
	\vskip 3pt
	
	\item The convergence analysis is challenged by approximation errors in the Dirichlet boundary data of the auxiliary Poisson problem. 
	Using discrete sine expansions, we establish an $L^2$ estimate for the corresponding discrete Poisson equation, which serves as a cornerstone for the convergence proof.
\end{itemize}
\vskip 3pt

The remainder of this paper is organized as follows. 
Section 2 presents the augmented HOC scheme for three-dimensional biharmonic equations. 
Section 3 discusses the convergence analysis. 
Section 4 reports numerical experiments demonstrating accuracy, efficiency, and applicability. 
Concluding remarks are given in Section 5.

\section{The augmented HOC solver}
In this section, we present a matrix-free augmented HOC difference solver for the biharmonic equation with variable coefficients~\cite{davoli2023spectral,engel2002continuous}:
\begin{equation}
	- \Delta(\alpha\Delta u)+\beta \Delta u=f,\quad \text{in}\  \Omega \subset \mathbb{R}^d,\ \  d=1,2,3,
\end{equation}
subject to the clamped boundary conditions:
\begin{subequations}\label{eq-bc}
	\begin{align}
		u &= g_D,
		\quad \text{on } \partial\Omega,
		\label{eq-bc-1}\\
		u_{\bf n} &= g_N,
		\quad \text{on } \partial\Omega_1;
		\qquad
		\Delta u = g_L,
		\quad \text{on } \partial\Omega_2.
		\label{eq-bc-2}
	\end{align}
\end{subequations}
where $\partial\Omega=\partial\Omega_1\cup\partial\Omega_2$, and $\partial\Omega_2$ may be empty. Here, $g_N$ and $g_L$ denote the Neumann-type and Laplacian-type boundary conditions, respectively.
To ensure ellipticity of the operator, we assume that the coefficient $\alpha({\bf x})$ is uniformly positive, namely,
$\alpha({\bf x})\ge \alpha_0>0,$
for some constant $\alpha_0$. In addition, we assume that $\alpha$ and $\beta$ are sufficiently smooth in $\Omega$, so that the fourth-order compact discretization is well defined and its truncation error analysis is valid. High-order methods demand stronger regularity than is required for well-posedness alone. We thus assume $f\in C^m(\Omega)$, $\alpha,\beta\in C^{m+2}(\Omega)$, and $u\in C^{m+4}(\Omega)$ for some $m\ge 0$. Note that  the method to be developed can  also be applied  to less regular solutions, though the convergence order may degrade accordingly.

Let $\Omega = [x_L,x_R]\times [y_L,y_R]\times [z_L,z_R]\subset\mathbb{R}^3$ be a hexahedral domain. We introduce a uniform Cartesian grid $\Omega_h$ on $\Omega$, consisting of points
\begin{equation}
	x_i = x_L + i h_x, \quad y_j = y_L + j h_y, \quad z_k = z_L + k h_z,
\end{equation}
for $i=0,\dots,N_x$, $j=0,\dots,N_y$, $k=0,\dots,N_z$, where the mesh spacings are
\[
h_x = \frac{x_R - x_L}{N_x}, \quad
h_y = \frac{y_R - y_L}{N_y}, \quad
h_z = \frac{z_R - z_L}{N_z}.
\]

\subsection{Augmented method}

Introducing the auxiliary variable $v=\alpha \Delta u$, we rewrite the biharmonic equation in the following mixed formulation:
\begin{align}
	\alpha\Delta u &= v, \quad \text{in } \Omega,\\
	-\Delta v+\frac{\beta}{\alpha}v &= f, \quad \text{in } \Omega,
\end{align}
supplemented by the clamped boundary conditions \eqref{eq-bc}.

However, on the boundary $\partial \Omega_1$, the value of $v$ is only implicitly determined by $\Delta u$. Consequently, the system is not closed and cannot be solved directly. To overcome this difficulty, we devise a high-order augmented strategy.

Taking the left boundary $x = x_L$ as $\partial\Omega_1$, we define an {\it augmented variable} $q(y,z) = v(x_L,y,z)$ on this boundary,
which resides solely on the boundary and thus has one fewer spatial dimension than the primary unknown $u$.
This yields the augmented equation
\begin{equation}
	\alpha \Delta u = q, \quad \text{on } \partial\Omega_1.
\end{equation}

It should be emphasized that the augmented variable is not an additional modeling unknown; rather, it represents the otherwise unavailable boundary values of $v$.
Any solution of the original mixed formulation satisfies the augmented system. 
Conversely, the augmented constraint ensures consistency between the auxiliary boundary values and the solutions of the two second-order subproblems, 
so that any solution of the augmented system recovers a solution of the original mixed problem. 
The augmented mixed formulation is therefore equivalent to the original mixed formulation.

Consequently, we obtain the following augmented mixed formulation, which consists of three coupled subsystems:

\noindent\textit{(i) The Poisson equation for $u$:}
\begin{equation}
	\left\{
	\begin{aligned}
		\Delta u &= \frac{v}{\alpha}, && \text{in } \Omega,\\
		u &= g_D, && \text{on } \partial\Omega. 
	\end{aligned}
	\right.
\end{equation}

\noindent\textit{(ii) The Helmholtz equation for $v$:}
\begin{equation}
	\left\{
	\begin{aligned}
		-\Delta v + \frac{\beta}{\alpha}v &= f, && \text{in } \Omega,\\
		v &= q, && \text{on } \partial\Omega_1,\\
		v &= g_L, && \text{on } \partial\Omega_2.
	\end{aligned}
	\right.
\end{equation}

\noindent\textit{(iii) The augmented equation:}
\begin{equation}
	\left\{
	\begin{aligned}
		q &= \alpha\Delta u, && \text{on } \partial\Omega_1,\\
		u_{\bf n} &= g_N, && \text{on } \partial\Omega_1.
	\end{aligned}
	\right.
\end{equation}

We next discretize the augmented mixed formulation via a high-order compact scheme and develop an efficient solver for the resulting algebraic system.

\subsection{HOC finite difference discretization}
We first introduce several discrete operators that will be employed in the discretization of the augmented mixed formulation.
For the approximation of the Laplace operator, we define the operators $L_{k\pm1}$ and $L_k$ as follows:
\begin{equation}
	L_{k\pm1}=\frac{1}{2}
	\begin{bmatrix}
		0 & \frac{1}{h_y^2}+\frac{1}{h_z^2} & 0\\[2pt]
		\frac{1}{h_x^2}+\frac{1}{h_z^2} & \frac{8}{h_z^2}-\frac{2}{h_x^2}-\frac{2}{h_y^2} & \frac{1}{h_x^2}+\frac{1}{h_z^2}\\[2pt]
		0 & \frac{1}{h_y^2}+\frac{1}{h_z^2} & 0
	\end{bmatrix},
\end{equation}
and
\begin{equation}
	L_k =\frac{1}{2}
	\begin{bmatrix}
		\frac{1}{h_x^2}+\frac{1}{h_y^2} & \frac{8}{h_y^2}-\frac{2}{h_x^2}-\frac{2}{h_z^2} & \frac{1}{h_x^2}+\frac{1}{h_y^2} \\[2pt]
		\frac{8}{h_x^2}-\frac{2}{h_y^2}-\frac{2}{h_z^2} & -\frac{16}{h_x^2}-\frac{16}{h_y^2}-\frac{16}{h_z^2} & \frac{8}{h_x^2}-\frac{2}{h_y^2}-\frac{2}{h_z^2} \\[2pt]
		\frac{1}{h_x^2}+\frac{1}{h_y^2} & \frac{8}{h_y^2}-\frac{2}{h_x^2}-\frac{2}{h_z^2} & \frac{1}{h_x^2}+\frac{1}{h_y^2}
	\end{bmatrix}.
\end{equation}

The fourth-order compact discrete Laplace operator is then given by 
\begin{equation*}
	L_h = \left[L_{k-1}, L_k, L_{k+1}\right].
\end{equation*}
We also define a weighting operator $M=[M_{k-1},M_k,M_{k+1}]$ as follows:
\begin{equation}
	M_{k\pm1}=\frac{1}{12}
	\begin{bmatrix}
		0 & 0 & 0\\
		0 & 1 & 0\\
		0 & 0 & 0
	\end{bmatrix}, \qquad
	M_{k}=\frac{1}{12}
	\begin{bmatrix}
		0 & 1 & 0\\
		1 & 6 & 1\\
		0 & 1 & 0
	\end{bmatrix}.
\end{equation}

For a grid function $\{W_{ijk}\}$, we define the local stencil matrix
\begin{equation*} 
	{\bf W}_{k+r}^{(i,j)} =
	\begin{bmatrix} 
		W_{i-1,j+1,k+r} & W_{i,j+1,k+r} & W_{i+1,j+1,k+r}\\ 
		W_{i-1,j,k+r} & W_{i,j,k+r} & W_{i+1,j,k+r}\\ 
		W_{i-1,j-1,k+r} & W_{i,j-1,k+r} & W_{i+1,j-1,k+r} 
	\end{bmatrix}, \qquad r=-1,0,1.
\end{equation*}
The action of $L_h$ at $(x_i,y_j,z_k)$ is then defined by 
\begin{equation}\label{eq-Lh-action} 
	L_hW_{ijk} = \sum_{r=-1}^{1} \langle L_{k+r},{\bf W}_{k+r}^{(i,j)}\rangle_F, 
\end{equation} 
where $\langle A,B\rangle_F = \sum_{p,q} A_{pq}B_{pq}$ denotes the Frobenius inner product. 
Similarly, the weighting operator $M$ acts on $W_{ijk}$ as
\begin{equation}\label{eq-M-action} 
	MW_{ijk} = \sum_{r=-1}^{1} \langle M_{k+r},{\bf W}_{k+r}^{(i,j)}\rangle_F.
\end{equation}
The fourth-order compact discretization of the equations for $u$ and $v$ is then given by
\begin{align}
	L_hU_{ijk} &= M(V/\alpha)_{ijk},\\
	-L_hV_{ijk} + M(\beta V/\alpha)_{ijk} &= Mf_{ijk},
\end{align}
where $U_{ijk}$ and $V_{ijk}$ denote the approximate values of $u$ and $v$ at the grid point $(x_i,y_j,z_k)$.
For sufficiently small mesh sizes $h$, the discrete system satisfies the discrete maximum principle.

We now consider the discretization of the augmented equation.
The main challenge lies in constructing a high-order approximation of the Laplacian of $u$ on the boundary.
Let ${\bf n}$ denote the unit outward normal vector to $\partial\Omega_1$, and let ${\boldsymbol\tau}_1$ and ${\boldsymbol\tau}_2$ be two mutually orthogonal unit tangent vectors. On $\partial\Omega_1$, the Laplacian of $u$ can then be decomposed as
\begin{equation}
	\Delta u = u_{{\bf n}{\bf n}} + u_{{\boldsymbol\tau}_1{\boldsymbol\tau}_1} + u_{{\boldsymbol\tau}_2{\boldsymbol\tau}_2}.
\end{equation}
The tangential components can be computed exactly from the Dirichlet boundary condition $u = g_D$.

Consider a grid point $(x_0,y_j,z_k)\in\partial\Omega_1$ on the boundary face
$x=x_0$. Applying the method of undetermined coefficients~\cite{li2023high} together with the Neumann boundary condition
$u_{\bf n}=g_N$, we construct the following high-order one-sided approximation $(\delta_{xx}U)_{0jk}$
to the second normal derivative $u_{{\bf nn}}(x_0,y_j,z_k)$:
\begin{equation*}
	(\delta_{xx}U)_{0jk}=
	\frac{1}{h^2}
	\left(
	-\frac{85}{18}U_{0jk}
	+6U_{1jk}
	-\frac{3}{2}U_{2jk}
	+\frac{2}{9}U_{3jk}
	\right)
	+\frac{11}{3h}g_N(x_0,y_j,z_k).
\end{equation*}
Analogous formulas can be derived for other boundaries where clamped conditions are prescribed.

\begin{remark}
	Although the formulation is presented for a single clamped face, it extends
	facewise to multiple clamped faces, including the fully clamped case. 
	The unknown trace \(\Delta u\) is introduced only at the relative interior
	grid points of each clamped face.
	At the edges and corners, all components of
	\(\Delta u\) can be obtained by differentiating the Dirichlet boundary conditions tangentially
	along adjacent faces.
	Hence, no augmented variables or equations are required there.
\end{remark}

Collecting all discrete equations yields the following coupled system in block matrix form:
\begin{equation}\label{eq-coupled}
	\begin{bmatrix}
		A_u      & B_u      & {\bf 0} \\
		{\bf 0}  & A_v      & C_v \\
		A_q      & {\bf 0}  & I_q
	\end{bmatrix}
	\begin{bmatrix}
		{\bf U} \\
		{\bf V} \\
		{\bf Q}
	\end{bmatrix}
	=
	\begin{bmatrix}
		{\bf 0} \\
		D_v{\bf F}_v \\
		{\bf F}_q
	\end{bmatrix}.
\end{equation}
Here, $A_u$ and $A_v$ are the discrete matrices corresponding to the operators $\Delta$ and $-\Delta + \beta/\alpha$, respectively, over the interior of the domain. 
The matrices $B_u$, $C_v$, and $D_v$ are weighting matrices with entries that are uniformly bounded with respect to the mesh size. 
The matrix $A_q$ contains the coefficients from the boundary approximations of the second-order normal derivatives, obtained via one-sided interpolation formulas. 
The vectors ${\bf F}_v$ and ${\bf F}_q$ collect the contributions from the source term and the prescribed boundary data, respectively. 
The identity matrix $I_q$ acts on the augmented variable vector ${\bf Q}$, all zero blocks are understood to have compatible dimensions.

The discretization described above constitutes only the first step of the proposed framework; the overall algorithmic design is not yet complete. 
In the next subsection, we develop an efficient solver tailored to the coupled discrete system, thereby realizing the full efficiency of the method.

\subsection{Efficient implementation of the augmented HOC discretization}
Thus far, we have focused on the derivation of the augmented HOC discretization; however, the overall efficiency of the method crucially depends on the solution strategy for the resulting algebraic system. In this subsection, we present a fast algorithm for the efficient implementation of the proposed scheme. The algorithm proceeds in two main steps: (i) computing the augmented variable ${\bf Q}$, and (ii) solving the two resulting generalized Poisson equations.

We begin by rewriting the coupled system \eqref{eq-coupled} and partitioning its coefficient matrix into four blocks:
\begin{equation}
	A_{11}=\begin{bmatrix}
		A_u & B_u\\
		{\bf 0} & A_v
	\end{bmatrix}, \quad
	A_{12}=\begin{bmatrix}
		{\bf 0}\\
		C_v
	\end{bmatrix},\quad
	A_{21}=\begin{bmatrix}
		A_q & {\bf 0}
	\end{bmatrix},\quad A_{22} = \begin{bmatrix}
		I_q
	\end{bmatrix}.
\end{equation}
The first step is to derive the Schur complement system for ${\bf Q}$, which is given by
\begin{equation}\label{eq-schur}
	S_h{\bf Q} = b,
\end{equation}
where $S_h = A_{22}-A_{21}A_{11}^{-1}A_{12}$ and $b=F_2-A_{21}A_{11}^{-1}F_1.$

While the Schur complement system could in principle be solved directly, the explicit assembly of $S_h$ and $b$ is computationally expensive, particularly as the grid is refined. To address this, we establish the following lemma, which plays a central role in developing a matrix-free solver for the augmented variable $\mathbf{Q}$.

\begin{lemma}\label{lemma-Re}
	For the augmented variable vector ${\bf Q}\in\mathbb{R}^{N_q}$ and its associated numerical solution ${\bf U}({\bf Q})$, the right-hand side of the Schur complement system is given by ${\bf b} = -{\bf R}({\bf 0})$, and the matrix-free action of $S_h$ on ${\bf Q}$ evaluates to
	\begin{equation}\label{eq-schur-action}
		S_h{\bf Q} = {\bf R}({\bf Q}) - {\bf R}({\bf 0}).
	\end{equation}
	Here, ${\bf R}({\bf Q}) \in \mathbb{R}^{N_q}$ is the discrete boundary residual vector with components
	\begin{equation*}\label{eq-boundary-residual-entry}
		R_{jk}({\bf Q}) = \alpha_{0jk} \left[ \left(\delta_{xx}U({\bf Q})\right)_{0jk} 
		+ \frac{\partial^2 g_D}{\partial y^2}(y_j,z_k) 
		+ \frac{\partial^2 g_D}{\partial z^2}(y_j,z_k) \right] - Q_{jk}, 
		\quad (j,k)\in\mathcal{I}_{\Gamma},
	\end{equation*}
	where $\mathcal{I}_{\Gamma}$ indexes the grid points on the boundary face $\partial\Omega_1$, and $(\delta_{xx}U({\bf Q}))_{0jk}$ is computed via the aforementioned high-order one-sided approximation.
\end{lemma}

\begin{proof}
	For a given vector ${\bf Q}$, we define $\mathcal{U}({\bf Q})$ as the solution to the linear system
	\begin{equation}
		A_{11}\mathcal{U}({\bf Q}) = F_1 - A_{12}{\bf Q}.
	\end{equation}
	In particular, when ${\bf Q} = {\bf 0}$, this reduces to
	\begin{equation}
		\mathcal{U}({\bf 0}) = A_{11}^{-1}F_1.
	\end{equation}
	
	From the definition of the discrete boundary residual vector, we have
	\begin{equation}\label{eq-algebraic-residual}
		{\bf R}({\bf Q}) = A_{22}{\bf Q} + A_{21}\mathcal{U}({\bf Q}) - F_2.
	\end{equation}
	Consequently, 
	\begin{align*}
		{\bf R}({\bf Q}) - {\bf R}({\bf 0}) 
		&= A_{22}{\bf Q} + A_{21}\mathcal{U}({\bf Q}) - A_{21}\mathcal{U}({\bf 0}) \\
		&= A_{22}{\bf Q} + A_{21}A_{11}^{-1}\left(F_1 - A_{12}{\bf Q}\right) - A_{21}A_{11}^{-1}F_1 \\
		&= \left(A_{22} - A_{21}A_{11}^{-1}A_{12}\right){\bf Q}.
	\end{align*}
	
	Recalling the definition $S_h = A_{22} - A_{21}A_{11}^{-1}A_{12}$, we obtain $S_h{\bf Q} = {\bf R}({\bf Q}) - {\bf R}({\bf 0})$. A similar substitution directly yields ${\bf b} = -{\bf R}({\bf 0})$, which completes the proof.
\end{proof}

Lemma \ref{lemma-Re} establishes that the action of $S_h$ on a vector ${\bf Q}$ can be evaluated in a matrix-free manner. Because Krylov subspace methods only require the matrix-vector product, this property makes GMRES an ideal choice for iteratively solving the Schur complement system.

During the GMRES iterations, evaluating the boundary residuals involves merely local interpolation operations, keeping the computational overhead minimal. Once a converged solution ${\bf Q}^m$ is obtained, the discrete variables ${\bf V}$ and ${\bf U}$ are recovered by sequentially solving the respective subproblems for $v$ and $u$ via FFT-based Poisson solvers. This completes the numerical solution of the fourth-order system on a compact stencil. Notably, the proposed method bypasses the assembly and storage of large sparse matrices entirely, relying exclusively on fast Poisson solvers. The complete procedure is summarized in Algorithm~\ref{alg:aug-hoc}.

\begin{algorithm}[htbp]
	\caption{Augmented HOC method with FFT-based solvers}
	\label{alg:aug-hoc}
	
	\SetKwInOut{Input}{Input}
	\SetKwInOut{Output}{Output}
	\SetKwBlock{Main}{Main procedure}{End of main procedure}
	\SetKwProg{Proc}{Sub-Procedure}{:}{End}
	
	\footnotesize
	\BlankLine
	
	\Main{
		\Input{Grid $\Omega_h$, boundary conditions, problem parameters, and stopping tolerance $\mathrm{tol}$.}
		\Output{Fourth-order numerical solutions ${\bf V}$ and ${\bf U}$.}
		\BlankLine
		
		${\bf b} \leftarrow \textsc{ConstructSchurRHS}(\;)$\;
		
		Apply matrix-free GMRES to solve $S_h{\bf Q}={\bf b}$ to the specified $\mathrm{tol}$, 
		evaluating each matrix-vector product via $\textsc{SchurMatVec}({\bf Q},{\bf b})$, 
		to obtain the converged augmented variable ${\bf Q}^{m}$\;
		
		Compute ${\bf V}$ by solving the discrete $v$-equation with ${\bf Q}^{m}$ via the FFT-based solver\;
		
		Compute ${\bf U}$ by solving the discrete $u$-equation with ${\bf V}$ via the FFT-based solver\;
	}
	
	\BlankLine
	\BlankLine
	
	\Indp
	\Proc{\textsc{ConstructSchurRHS}\textnormal{()}}{
		\Input{Residual operator ${\bf R}(\cdot)$.}
		\Output{The Schur right-hand side vector ${\bf b}$.}
		\BlankLine
		
		Set ${\bf Q}\leftarrow{\bf 0}$\;
		
		Compute ${\bf V}({\bf 0})$ by solving the discrete $v$-equation with ${\bf Q}={\bf 0}$ via the FFT-based solver\;
		
		Compute ${\bf U}({\bf 0})$ by solving the discrete $u$-equation with ${\bf V}({\bf 0})$ via the FFT-based solver\;
		
		Evaluate the residual vector ${\bf R}({\bf 0})$ according to Lemma~\ref{lemma-Re}\;
		
		${\bf b}\leftarrow-{\bf R}({\bf 0})$\;
		
		\textbf{return} ${\bf b}$\;
	}
	\Indm
	
	\BlankLine
	\BlankLine
	
	\Indp
	\Proc{\textsc{SchurMatVec}\textnormal{(${\bf Q},{\bf b}$)}}{
		\Input{A given vector ${\bf Q}$ and the Schur right-hand side vector ${\bf b}$.}
		\Output{The matrix-vector product $S_h{\bf Q}$.}
		\BlankLine
		
		Compute ${\bf V}({\bf Q})$ by solving the discrete $v$-equation with ${\bf Q}$ via the FFT-based solver\;
		
		Compute ${\bf U}({\bf Q})$ by solving the discrete $u$-equation with ${\bf V}({\bf Q})$ via the FFT-based solver\;
		
		Evaluate the residual vector ${\bf R}({\bf Q})$ according to Lemma~\ref{lemma-Re}\;
		
		${\bf Y}\leftarrow {\bf R}({\bf Q})+{\bf b}$ \tcp*{Since ${\bf b} = -{\bf R}({\bf 0})$, this computes $S_h{\bf Q}$}
		
		\textbf{return} ${\bf Y}$\;
	}
	\Indm
	
\end{algorithm}

\begin{remark}
	Let \(n=N_xN_yN_z\) denote the total number of grid points and
	\(n_{\Gamma}\) the number of augmented boundary unknowns. For a fixed
	number of augmented boundary faces, \(n_{\Gamma}=O(N^2)\) when
	\(N_x=N_y=N_z=N\), whereas \(n=O(N^3)\).
	Each constant-coefficient subproblem is solved by a fast direct solver
	based on discrete sine transforms at a cost of \(O(n\log n)\). For
	variable-coefficient problems, the \({\bf V}\)'s subproblem is solved by
	matrix-free BiCGSTAB. Hence, if \(J_{\mathrm{out}}\) and \(J_{\mathrm{in}}\) denote
	the outer GMRES iteration count and the average inner BiCGSTAB iteration
	count, respectively, the dominant computational cost is estimated as
	\[
	O\!\left(
	(J_{\mathrm{out}}+1)(J_{\mathrm{in}}+1)n\log n
	\right).
	\]
	Thus, the effective computational cost is \(O(n\log n)\), provided that
	the inner and outer iteration counts remain bounded under grid refinement, as observed in the numerical experiments. 
	This is a conditional cost estimate rather than a uniform
	complexity bound.
\end{remark}

\section{Convergence analysis}

In this section, we establish the convergence of the augmented HOC method for the classical biharmonic equation. This restriction is imposed solely to simplify the presentation of the proof, as the variable-coefficient case involves cumbersome technical details. Nevertheless, the theoretical framework remains valid for a spatially varying $\alpha$ provided $\beta=0$. Our main convergence result is stated as follows.

\begin{theorem}\label{thm-converge}
	Suppose $u\in C^6({\Omega})$ is the exact solution of $\Delta^2 u=f$ with clamped boundary conditions on $\partial\Omega$. Let ${\bf U}$ be the numerical solution of the augmented HOC scheme on $\Omega_h$, and let ${\bf u}_h=\bigl(u({\bf x})\bigr)_{{\bf x}\in\Omega_h}$ be its restriction to the grid. If the error of the augmented variable satisfies $\|{\bf E}^q\|_{2}={O}(h^{7/2})$, then there exists a constant $C>0$, independent of $h$, such that
	\begin{equation}\label{eq-u-convergence}
		\|{\bf E}^u\|_{\infty} = \|{\bf U}-{\bf u}_h\|_{\infty} \leq C h^{7/2}.
	\end{equation}
\end{theorem}

Before presenting the proof of Theorem~\ref{thm-converge}, we introduce two lemmas that play a central role in our analysis. The proof of Lemma~\ref{thm-infStab} is standard and can be found in classical literature. In contrast, Lemma~\ref{thm-estbound} constitutes the main technical contribution established in this work.

\begin{lemma}\label{thm-infStab}
	Let $\Omega\subset\mathbb{R}^d$, $d=2,3$, be a rectangular domain equipped
	with a uniform Cartesian grid $\Omega_h$. Suppose that the grid function
	$\Phi=\{\phi_{\boldsymbol{i}}\}$ and ${\bf G}=\{g_i\}$ satisfies
	\begin{equation}
		L_h\phi_{\boldsymbol{i}}
		=
		g_{\boldsymbol{i}},
		\qquad
		g_{\boldsymbol{i}}:=Mf_{\boldsymbol{i}},
		\qquad
		\boldsymbol{x}_{\boldsymbol{i}}\in\Omega_h^0,
	\end{equation}
	with $\phi_{\boldsymbol{i}}=0$ on $\partial\Omega_h$. Then there exists a
	constant $C>0$, independent of $h$, such that
	\begin{equation}
		\|\Phi\|_{\infty}
		\leq
		C\|{\bf G}\|_{2}.
	\end{equation}
\end{lemma}

\begin{remark}
	The proof of Lemma~\ref{thm-infStab} is standard in both two and three dimensions. The main ingredients are the discrete elliptic regularity of the compact operator $L_h$ and the discrete Sobolev embedding from $H_h^2(\Omega_h)$ into $\infty(\Omega_h)$. More precisely, for grid functions vanishing on $\partial\Omega_h$, one first uses the symmetry and positive definiteness of $L_h$ to obtain
	\[
	\|\Phi\|_{H^2}\le C_1\|{\bf G}\|_{2},
	\]
	and then applies the discrete Sobolev embedding
	\[
	\|\Phi\|_{\infty}\le C_2\|\Phi\|_{H^2}.
	\]
	Combining these two estimates yields
	\[
	\|\Phi\|_{\infty}\le C\|{\bf G}\|_{2},
	\]
	where the constant $C$ is independent of $h$. Since the same two ingredients remain valid on uniform Cartesian grids in both two and three dimensions, the proof is identical up to notation.
\end{remark}

\begin{lemma}\label{thm-estbound}
	Let $\Omega\subset\mathbb{R}^d$ \textnormal{($d=2,3$)} be a rectangular domain, and let $\Omega_h$ be a uniform Cartesian grid with interior grid set $\Omega_h^0$. Let $\Psi=\{\psi_{\boldsymbol{i}}\}$ satisfy
	\begin{equation}
		L_h\psi_{\boldsymbol{i}} = 0, \quad \boldsymbol{x}_{\boldsymbol{i}}\in \Omega_h^0,
	\end{equation}
	subject to the boundary conditions
	\begin{equation}
		\psi_{\boldsymbol{i}} = b_{\boldsymbol{i}}, \quad \boldsymbol{x}_{\boldsymbol{i}}\in \Gamma_h\quad
		\text{and}\quad
		\psi_{\boldsymbol{i}} = 0, \quad \boldsymbol{x}_{\boldsymbol{i}}\in \partial\Omega_h\setminus\Gamma_h,
	\end{equation}
	where $\Gamma_h$ denotes one boundary edge for $d=2$ and one boundary face for $d=3$, and $L_h$ denotes the corresponding fourth-order compact discrete Laplace operator in $\mathbb{R}^d$.
	
	Then there exists a constant $C>0$, independent of $h$, such that
	\begin{equation}
		\|\Psi\|_{2} \le C \|{\bf b}\|_{2(\Gamma_h)},
	\end{equation}
	where ${\bf b}=\{b_{\boldsymbol{i}}\}$.
\end{lemma}

% The proof of Theorem~\ref{thm-estbound} depends on the spatial dimension. 
% We therefore present the detailed proof in the following subsection.

To prove Lemma~\ref{thm-estbound}, we first establish several auxiliary lemmas and corollary in a unified framework for $(d=2,3)$.
Their proofs are deferred to the appendix.

\begin{lemma}\label{lem-repre-ND}
	For $d=2$ or $3$, suppose that the same conditions as in
	Lemma~\ref{thm-estbound} are satisfied. We write the grid index as
	$(i,\boldsymbol m)$, where $i$ denotes the index in the $x$-direction and
	\[
	\boldsymbol m=(m_1,\ldots,m_{d-1})
	\in \{1,\ldots,N-1\}^{d-1}
	\]
	denotes the indices in the remaining directions, e.g., $m_1:=j$ in the y-direction. Then the solution
	$\psi_{i,\boldsymbol m}$ admits the representation
	\begin{equation}\label{eq-varphi-repre-ND}
		\psi_{i,\boldsymbol m}
		=
		\sum_{\boldsymbol\ell\in\{1,\cdots,N-1\}^{d-1}}
		\widehat g_{\boldsymbol\ell}\,
		r_{i,\boldsymbol\ell}
		\prod_{s=1}^{d-1}
		\sin\frac{\ell_s\pi m_s}{N},
	\end{equation}
	where
	\begin{equation}
		\widehat g_{\boldsymbol\ell}
		=
		\left(\frac{2}{N}\right)^{d-1}
		\sum_{\boldsymbol m\in\{1,\ldots,N-1\}^{d-1}}
		g_{\boldsymbol m}
		\prod_{s=1}^{d-1}
		\sin\frac{\ell_s\pi m_s}{N}.
	\end{equation}
	For each fixed transverse mode $\boldsymbol\ell$, the coefficient
	$r_{i,\boldsymbol\ell}$ satisfies a one-dimensional recurrence relation
	in the $x_1$-direction. In particular,
	\begin{equation}
		r_{i,\boldsymbol\ell}
		=
		\frac{\sinh((N-i)\alpha_{\boldsymbol\ell})}
		{\sinh(N\alpha_{\boldsymbol\ell})}.
	\end{equation}
	Consequently,
	\begin{equation}
		\psi_{i,\boldsymbol m}
		=
		\sum_{\boldsymbol\ell\in\{1,\ldots,N-1\}^{d-1}}
		\widehat g_{\boldsymbol\ell}\,
		\frac{\sinh((N-i)\alpha_{\boldsymbol\ell})}
		{\sinh(N\alpha_{\boldsymbol\ell})}
		\prod_{s=1}^{d-1}
		\sin\frac{\ell_s\pi m_s}{N}.
	\end{equation}
	Moreover, by the discrete orthogonality of sine functions,
	\begin{equation}
		\|\Psi\|_{2(\Omega_h^0)}^2
		=
		\frac{1}{2^{d-1}}
		\sum_{\boldsymbol\ell\in\{1,\ldots,N-1\}^{d-1}}
		|\widehat g_{\boldsymbol\ell}|^2
		\sum_{i=1}^{N-1} r_{i,\boldsymbol\ell}^2\,h .
	\end{equation}
\end{lemma}

\begin{lemma}\label{lem-boundmw-ND}
	Let $d=2$ or $3$. There exists a constant $C>0$, independent of
	$h$ and the transverse mode index
	$\boldsymbol\ell\in\{1,\ldots,N-1\}^{d-1}$, such that
	\begin{equation}
		\sum_{i=1}^{N-1}
		r_{i,\boldsymbol\ell}^{\,2}\,h
		\le C.
	\end{equation}
\end{lemma}

\begin{corollary}\label{cor-dpi-ND}
	Let $d\in\{2,3\}$ and define the discrete sine coefficients by
	\begin{equation}
		\widehat g_{\boldsymbol\ell}
		=
		\left(\frac{2}{N}\right)^{d-1}
		\sum_{\boldsymbol m\in\{1,\ldots,N-1\}^{d-1}}
		g_{\boldsymbol m}
		\prod_{s=1}^{d-1}
		\sin\frac{\ell_s\pi m_s}{N},
		\qquad
		\boldsymbol\ell\in\{1,\ldots,N-1\}^{d-1}.
	\end{equation}
	Then the following discrete Parseval identity holds:
	\begin{equation}
		h^{d-1}
		\sum_{\boldsymbol m\in\{1,\ldots,N-1\}^{d-1}}
		|g_{\boldsymbol m}|^2
		=
		\frac{1}{2^{d-1}}
		\sum_{\boldsymbol\ell\in\{1,\ldots,N-1\}^{d-1}}
		|\widehat g_{\boldsymbol\ell}|^2.
	\end{equation}
\end{corollary}

In the following subsections, we provide the proofs of the convergence theorem for the two- and three-dimensional cases, respectively.

\begin{remark}
	For clarity, the convergence analysis is presented for a single clamped
	edge in two dimensions or face in three dimensions. 
	It extends to any fixed
	number of clamped boundary components, including the fully clamped case.
	Indeed, the augmented variable error is supported only in the relative
	interiors of these components. 
	By linearity, the error can be decomposed
	into single component contributions. Lemma~\ref{thm-estbound}, after a
	permutation of coordinates, applies to each contribution, while the remaining
	consistency and stability estimates follow in the same way. 
	Since the number
	of boundary components is independent of \(h\), the convergence rate remains
	unchanged.
\end{remark}

\subsection{Two-dimensional convergence analysis}
Since the proof in two dimensions is more concise, we begin with the two-dimensional case.
At first, we briefly introduce the corresponding two-dimensional discretization. 
In two dimensions, the fourth-order compact discrete Laplace operator is defined by
$
L_h=[L_{j-1},L_j,L_{j+1}]$,
where
\begin{equation}
	L_{j\pm1}=
	\begin{bmatrix}
		0 & \dfrac{1}{h_x^2}+\dfrac{1}{h_y^2} & 0
	\end{bmatrix},
	\qquad
	L_j=
	\begin{bmatrix}
		\dfrac{1}{h_x^2}+\dfrac{1}{h_y^2}
		&
		-\dfrac{4}{h_x^2}-\dfrac{4}{h_y^2}
		&
		\dfrac{1}{h_x^2}+\dfrac{1}{h_y^2}
	\end{bmatrix}.
\end{equation}
The associated weighting operator is given by
$M=[M_{j-1},M_j,M_{j+1}]$,
with
\begin{equation}
	M_{j\pm1}=\frac{1}{12}
	\begin{bmatrix}
		0 & 1 & 0
	\end{bmatrix},
	\qquad
	M_j=\frac{1}{12}
	\begin{bmatrix}
		1 & 8 & 1
	\end{bmatrix}.
\end{equation}
Within the above discrete operators, we can obtain the resulting two dimensional discrete system which has the same algebraic structure as \eqref{eq-coupled}. 
The corresponding augmented strategy is likewise similar to that in three dimensions, the only essential difference being that the augmented variable becomes $q(y)=v(x_L,y)$, as the z-direction is absent in two dimensions.

We next prove Lemma~\ref{thm-estbound} in 2D. 
The proof relies primarily on Lemmas~\ref{lem-repre-ND} and~\ref{lem-boundmw-ND} and Corollary~\ref{cor-dpi-ND} in 2D, whose proofs are deferred to the appendix.

\begin{proof} 
	By Lemma~\ref{lem-repre-ND} of two dimensional cases, $\Psi=\{\psi_{ij}\}$ admits the representation 
	\begin{equation}
		\psi_{i,j} = \sum_{k=1}^{N-1} \widehat g_k \frac{\sinh((N-i)\alpha_k)}{\sinh(N\alpha_k)} \sin\frac{k\pi j}{N}, \qquad 1\le i,j\le N-1, \end{equation} 
	where 
	\begin{equation} 
		\widehat g_k = \frac{2}{N}\sum_{m=1}^{N-1} g_m \sin\frac{k\pi m}{N}, \end{equation} 
	and
	\begin{equation}
		\cosh\alpha_k = \frac{5-2\cos\theta_k}{2+\cos\theta_k}, \qquad \theta_k=\frac{k\pi}{N}.
	\end{equation} 
	Then, it gives 
	\begin{equation} \|\psi\|_{2(\Omega_h^0)}^2 = \frac{1}{2} \sum_{k=1}^{N-1} |\widehat g_k|^2 \sum_{i=1}^{N-1} \left( \frac{\sinh((N-i)\alpha_k)}{\sinh(N\alpha_k)} \right)^2 h. \end{equation} 
	By Lemma \ref{lem-boundmw-ND}, \begin{equation} \sum_{i=1}^{N-1} \left( \frac{\sinh((N-i)\alpha_k)}{\sinh(N\alpha_k)} \right)^2 h \le C, 
	\end{equation} 
	and by Corollary \ref{cor-dpi-ND}, 
	\begin{equation} \sum_{k=1}^{N-1}|\widehat g_k|^2 = 2h\sum_{j=1}^{N-1}|g_j|^2. 
	\end{equation} 
	Combining the above estimates and $\psi_{0,j}=g_j$, we arrive at 
	\begin{equation} 
		\|\Psi\|_{2}^2 \le C h \sum_{j=1}^{N-1}|g_j|^2.
	\end{equation} 
\end{proof}
We now proceed to the convergence analysis of the augmented HOC scheme in two dimensions.

\begin{proof}
	For convenience, define the errors by
	\[
	{\bf E}^u={\bf u}_h-{\bf U},\qquad
	{\bf E}^v={\bf v}_h-{\bf V},\qquad
	{\bf E}^q={\bf q}_h-{\bf Q},
	\]
	where ${\bf u}_h$, ${\bf v}_h$, and ${\bf q}_h$ denote the restrictions of
	the exact solutions $u$, $v$, and $q$ to the corresponding grids.
	Substituting the exact solution into the discrete system and subtracting the
	numerical equations yield
	\begin{equation}\label{eq-error-system}
		A_u{\bf E}^u+B_u{\bf E}^v={\bf T}^u,
		\qquad
		A_v{\bf E}^v+C_v{\bf E}^q={\bf T}^v,
	\end{equation}
	where ${\bf T}^u$ and ${\bf T}^v$ are the truncation-error vectors associated
	with the interior equations for $u$ and $v$, respectively. The consistency of
	the HOC discretization gives
	\[
	\|{\bf T}^u\|_2+\|{\bf T}^v\|_2\leq Ch^4.
	\]
	
	At the interior grid points, \eqref{eq-error-system} is equivalent to
	\[
	L_hE^u_{ij}=ME^v_{ij}+T^u_{ij},
	\qquad
	L_hE^v_{ij}=T^v_{ij}.
	\]
	Moreover, ${\bf E}^u$ satisfies homogeneous boundary conditions. Since
	$v=q$ and $V=Q$ on $\partial\Omega_1$, we have
	$E^v_{0j}=E^q_j$ for $1\leq j\leq N-1$, whereas ${\bf E}^v$ vanishes on
	the remaining boundary grid points.
	
	We decompose $ {\bf E}^v={\bf E}^{v,1}+{\bf E}^{v,2},$
	where ${\bf E}^{v,1}$ satisfies the homogeneous interior equation with
	boundary data ${\bf E}^q$ on $\partial\Omega_1$ and zero data on the
	remaining boundary, while ${\bf E}^{v,2}$ satisfies the inhomogeneous
	equation with source ${\bf T}^v$ and homogeneous boundary conditions.
	Lemma~\ref{thm-estbound} and the assumed estimate for the augmented variable
	give
	\[
	\|{\bf E}^{v,1}\|_2
	\leq C\|{\bf E}^q\|_2
	\leq Ch^{7/2}.
	\]
	By Lemma~\ref{thm-infStab},
	\[
	\|{\bf E}^{v,2}\|_2
	\leq C\|{\bf T}^v\|_2
	\leq Ch^4.
	\]
	Therefore,
	\[
	\|{\bf E}^v\|_2
	\leq
	\|{\bf E}^{v,1}\|_2+\|{\bf E}^{v,2}\|_2
	\leq Ch^{7/2}.
	\]
	
	Finally, applying Lemma~\ref{thm-infStab} to the equation for ${\bf E}^u$
	and using the uniform boundedness of $M$ in the discrete $L_2$ norm, we obtain
	\[
	\|{\bf E}^u\|_{\infty}
	\leq
	C\bigl(\|M{\bf E}^v\|_2+\|{\bf T}^u\|_2\bigr)
	\leq
	C\bigl(\|{\bf E}^v\|_2+\|{\bf T}^u\|_2\bigr)
	\leq Ch^{7/2}.
	\]
	Since
	$\|{\bf E}^u\|_{\infty}
	=\|{\bf U}-{\bf u}_h\|_{\infty}$,
	the proof is complete.
\end{proof}

The convergence analysis yields a conditional \(O(h^{7/2})\) error estimate, provided that the augmented boundary variables satisfy an \(O(h^{7/2})\) discrete \(L^2\)-error bound. Although a rigorous proof of this boundary estimate is not yet available, it is consistently supported by the numerical results, under this assumption, the remaining convergence analysis is rigorous. The resulting estimate does not recover the full fourth-order rate 
observed numerically. This loss of half an order appears to arise from a non-sharp treatment of the boundary contribution. The numerical results suggest that the theoretical bound is not optimal and that the actual convergence rate is fourth order.

\subsection{Three-dimensional extension}

The convergence analysis in three dimensions follows essentially the same
framework as that in the two-dimensional case. The main differences are that
the boundary forcing is supported on a boundary face rather than on a boundary
edge, and that a double discrete sine expansion is required in the two
transverse directions. We therefore present only the modifications needed for
the three-dimensional case.

As in the two-dimensional analysis, the proof of
Theorem~\ref{thm-converge} relies on the stability estimate stated in
Lemma~\ref{thm-estbound}. We first establish this estimate for $d=3$ by
extending the auxiliary results obtained in two dimensions.

Compared with the proof of Lemma~\ref{lem-repre-ND} in the
two-dimensional case, the only modification is that the discrete sine
expansion is performed in both transverse directions rather than in a single
direction. The remainder of the argument is unchanged and relies on the
orthogonality of the tensor-product discrete sine basis:
\begin{equation}
	h^2\sum_{j=1}^{N-1}\sum_{k=1}^{N-1}
	\sin\frac{p\pi j}{N}
	\sin\frac{q\pi k}{N}
	\sin\frac{p'\pi j}{N}
	\sin\frac{q'\pi k}{N}
	=
	\frac{1}{4}\delta_{pp'}\delta_{qq'}.
\end{equation}
Thus, the three-dimensional representation formula is the direct
tensor product analogue of Lemma~\ref{lem-repre-ND}.

The proof of Lemma~\ref{lem-boundmw-ND} is also identical to its
two-dimensional counterpart. It suffices to replace the one-dimensional mode
parameter $\alpha_p$ by the two-dimensional mode parameter
$\alpha_{p,q}$ and use the lower bound $\alpha_{p,q}
\gtrsim
\sqrt{p^2+q^2}/N,$
or any equivalent estimate sufficient to control the resulting geometric
series. No further modification of the argument is required.

Similarly, Corollary~\ref{cor-dpi-ND} follows by applying the
two-dimensional discrete Parseval identity in tensor-product form. More
precisely, it is an immediate consequence of the orthogonality of the double
discrete sine basis.

Combining the representation formula, the uniform bound for the modal
coefficients, and the discrete Parseval identity, we obtain
\begin{equation}
	\|{\Psi}\|_{2(\Omega_h^0)}^2
	\le
	C h^2
	\sum_{j=1}^{N-1}\sum_{k=1}^{N-1}
	|g_{j,k}|^2.
\end{equation}
This proves Lemma~\ref{thm-estbound} for $d=3$.

We now briefly outline the proof of Theorem~\ref{thm-converge} in three dimensions.
As in the two-dimensional case, the error ${\bf E}^v$ is decomposed into a
boundary-driven component and an interior-forcing component ${\bf E}^v
=
{\bf E}^{v,1}
+
{\bf E}^{v,2}.$
The boundary-driven component ${\bf E}^{v,1}$ is estimated by
Lemma~\ref{thm-estbound}, whereas the interior-forcing component
${\bf E}^{v,2}$ is controlled by Lemma~\ref{thm-infStab}. Therefore,
using $\|{\bf E}^q\|_{2}
\le
C h^{7/2}$
and the consistency estimate $\|T^v\|_{2}
\le
C h^4,$
we obtain $\|{\bf E}^v\|_{2}
\le
\|{\bf E}^{v,1}\|_{2}
+
\|{\bf E}^{v,2}\|_{2}
\le
C h^{7/2}.$

The estimate for ${\bf E}^u$ then follows from
Lemma~\ref{thm-infStab}. Indeed, using the equation satisfied by
${\bf E}^u$, together with the estimate for ${\bf E}^v$ and the consistency
bound $\|T^u\|_{2}
\le
C h^{4},$
we obtain the corresponding error estimate for ${\bf E}^u$. This completes
the proof of Theorem~\ref{thm-converge} in three dimensions.

The preceding analysis establishes a rigorous $O(h^{7/2})$ convergence
estimate in three dimensions. This estimate does not yet recover the
fourth-order convergence rate observed in the numerical experiments. The loss
of one-half order appears to arise from the estimate of the boundary-driven
component, which does not exploit possible cancellation or additional
structural properties of the boundary error. A sharper analysis of this term
may therefore lead to the optimal fourth-order estimate.

\section{Numerical experiments}
In this section, we assess the accuracy, efficiency, and robustness
of the proposed method through two sets of numerical experiments. The first
set employs problems with manufactured solutions to examine the convergence
order and computational efficiency in terms of CPU time. 
The second set
considers benchmark problems and higher-order PDEs to evaluate the robustness, and broader applicability of the method. 

We also describe the key
computational ingredients and solver settings used throughout the experiments.
In the convergence test presented below, the Schur complement system for the augmented variables is solved using \emph{a matrix-free, non-restarted GMRES iteration} equipped with a boundary-symbol preconditioner~\cite{toselli2005domain,benzi2005numerical}. 

For the two discrete finite difference subproblems arising from the mixed formulation of the biharmonic equation, if $\beta=0$, or in the constant-coefficient biharmonic case, both subproblems can be solved efficiently by FFT solvers. When $\beta\ne 0$, we still use an FFT solver for the $u$-equation, while the $v$-equation is solved by a matrix-free BiCGSTAB iteration preconditioned by an FFT-based solver.

\subsection{Convergence tests}

We first present numerical experiments with analytical solutions to assess the accuracy and convergence behavior of the proposed method. The convergence tests include both two- and three-dimensional problems with constant and variable coefficients. In these tests, clamped boundary conditions are imposed on the boundary $x=x_L$, while Navier boundary conditions are prescribed on the remaining boundaries. 
When solving the Schur complement system for the augmented variables, the GMRES stopping tolerance at each grid level is set proportional to $h^{3.5}$.

\begin{example}\label{ex-2d}
	We consider a two-dimensional test problem with the following manufactured solution,
	adapted from~\cite{meng2025adaptive}:
	\begin{equation*}
		u(x,y)
		=
		256\Big(x^2+\varepsilon\big(1-e^{-\frac{x}{\sqrt{\varepsilon}}}\big)^2\Big)
		(x-1)^2y^2(y-1)^2 .
	\end{equation*}
	We choose $\alpha=\varepsilon$ and $\beta=1$. When $\varepsilon$ is sufficiently
	small, the equation
	\begin{equation*}
		-\varepsilon\Delta^2 u+\Delta u=f
	\end{equation*}
	becomes a fourth-order singularly perturbed problem.
\end{example}
\begin{table}[htbp]
\footnotesize
\centering
\caption{Grid refinement analysis of the fourth-order compact augmented scheme for Example~\ref{ex-2d}.}
\label{tab-2d}
\begin{tabular}{|c||c|c|c||c|c|c||c|c|c|}
	\hline
	\multirow{2}{*}{$N_x\times N_y$}
	& \multicolumn{3}{c||}{$\varepsilon=10^{-4}$}
	& \multicolumn{3}{c|}{$\varepsilon=10^{-5}$} \\
	\cline{2-7}
	& $\|{\bf E}^u\|_{\infty}$ & Order & Iters
	& $\|{\bf E}^u\|_{\infty}$ & Order & Iters
	\\
	\hline
	$256\times256$
	& $1.72\mathrm{E}{-05}$ & --   & 4
	& $6.55\mathrm{E}{-05}$ & --   & 3 \\
	$512\times512$
	& $1.47\mathrm{E}{-06}$ & 3.55 & 2
	& $8.17\mathrm{E}{-06}$ & 3.00 & 3 \\
	$1024\times1024$
	& $1.09\mathrm{E}{-07}$ & 3.76 & 2
	& $8.03\mathrm{E}{-07}$ & 3.35 & 3 \\
	$2048\times2048$
	& $7.41\mathrm{E}{-09}$ & 3.87 & 2
	& $6.46\mathrm{E}{-08}$ & 3.64 & 3 \\
	$4096\times4096$
	& $4.54\mathrm{E}{-10}$ & 4.03 & 2
	& $4.65\mathrm{E}{-09}$ & 3.80 & 3 \\
	\hline
\end{tabular}
\end{table}

\begin{figure}[htbp]
\centering
\includegraphics[width=0.5\textwidth]{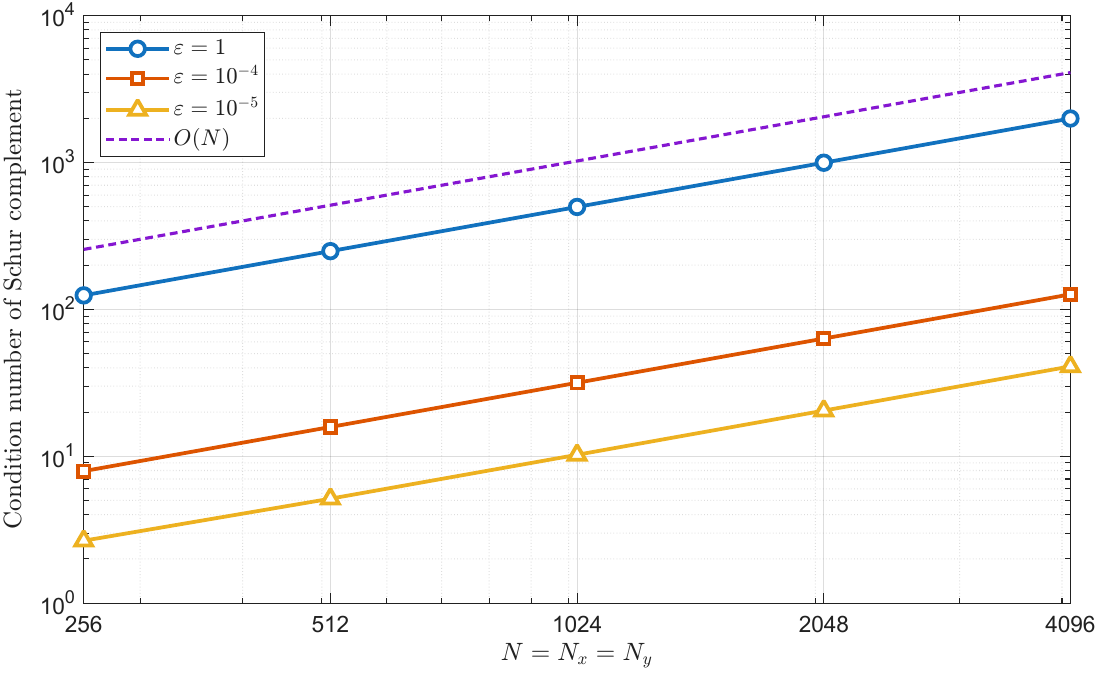}
\caption{Conditioning numbers in three values of $\varepsilon$.}
\label{fig-condition-number}
\end{figure}

Table~\ref{tab-2d} presents the numerical errors and convergence orders for Example~\ref{ex-2d}. 
For both $\varepsilon=10^{-4}$ and $\varepsilon=10^{-5}$,
the $L_\infty$ errors of $u$ decrease rapidly as the mesh is refined, and the
observed convergence rates gradually approach fourth order. The case
$\varepsilon=10^{-5}$ is more challenging because the boundary layer near $x=0$
is thinner, which explains the slightly lower convergence rates on relatively
coarse meshes. On sufficiently fine grids, the expected high-order convergence is
clearly recovered.

Here, at each refinement level, the solution obtained on the preceding coarse grid is interpolated onto the current fine grid and used as the initial guess for the GMRES iteration.
Thus, the GMRES iteration number remains small for all
mesh sizes, showing that the augmented solver is highly efficient for this
singularly perturbed problem.

Figure~\ref{fig-condition-number} presents the condition numbers of the Schur
complement systems as the grid size $N$ increases from $256$ to $4096$.
The observed $O(h^{-1})$ growth is relatively mild and is consistent with the
nearly mesh-independent GMRES iteration counts reported in the numerical
experiments.

\begin{example}\label{ex-3d-peak}
	We consider the following manufactured solution,
	\begin{equation*}
		u(x,y,z)=(x^2-x)(y^2-y)(z^2-z)e^{\,q[(x-0.5)^2+(y-0.5)^2+(z-p)^2]}.
	\end{equation*}
	This solution becomes strongly peaked for large values of the parameter $q$. The parameter $p$ controls the location of the peak in the $z$-direction.
\end{example}

Table~\ref{table-3d-peak} shows that the proposed augmented HOC finite difference scheme performs robustly for the variable-coefficient fourth-order problem with $\alpha=x^2+y^2+z^2+1$ and $\beta=x+y+z+10$. 
For both parameter settings, the error $\|{\bf E}^u\|_\infty$ decays monotonically under mesh refinement, and the observed convergence rates are close to fourth order on refined meshes. 
The case $(q,p)=(10,0.2)$ produces larger errors than the case $(q,p)=(5,1)$ on the same mesh, which is consistent with the fact that the corresponding exact solution exhibits stronger local variation. 
However, the asymptotic convergence behavior remains essentially unchanged. The GMRES iteration counts stay between $4$ and $10$ throughout, indicating that the augmented solve remains efficient even in the presence of variable coefficients.

Figure~\ref{fig-3d-peak-u2} displays cross sections of ${\bf U}$ for Example~\ref{ex-3d-peak} with $(q,p)=(5,1)$ on a $256\times256\times256$ grid. The slices at fixed $z$-, $y$-, and $x$-levels visualize the localized variation near the lower boundary.

\begin{table}[htbp]
	\footnotesize
	\centering
	\caption{Grid refinement analysis of the  augmented HOC finite difference scheme for Example~\ref{ex-3d-peak} with $\alpha=x^2+y^2+z^2+1$ and $\beta=x+y+z+10$.}
	\label{table-3d-peak}
	\begin{tabular}{|c||c|c|c||c|c|c|}
		\hline
		\multirow{2}{*}{$N_x\times N_y\times N_z$}
		& \multicolumn{3}{c||}{$(q,p)=(10,0.2)$}
		& \multicolumn{3}{c|}{$(q,p)=(5,1)$} \\
		\cline{2-7}
		& $||{\bf E}^u||_{\infty}$ & Order & Iters
		& $||{\bf E}^u||_{\infty}$ & Order & Iters \\
		\hline
		$32\times 32\times 32$
		& 5.36E-02 & --   & 5
		& 6.86E-04 & --   & 4 \\
		
		$64\times 64\times 64$
		& 3.43E-03 & 3.96 & 6
		& 4.35E-05 & 3.98 & 6 \\
		
		$128\times 128\times 128$
		& 2.17E-04 & 3.99 & 7
		& 2.73E-06 & 3.99 & 7 \\
		
		$256\times 256\times 256$
		& 1.36E-05 & 3.99 & 8
		& 1.71E-07 & 4.00 & 8 \\
		
		$512\times 512\times 512$
		& 8.50E-07 & 3.99 & 10
		& 1.07E-08 & 4.00 & 9 \\
		\hline
	\end{tabular}
\end{table}

\begin{figure}[tbhp]
	\centering
	\subfloat[Slices  at fixed $z$-levels.]{\includegraphics[width=0.33\textwidth]{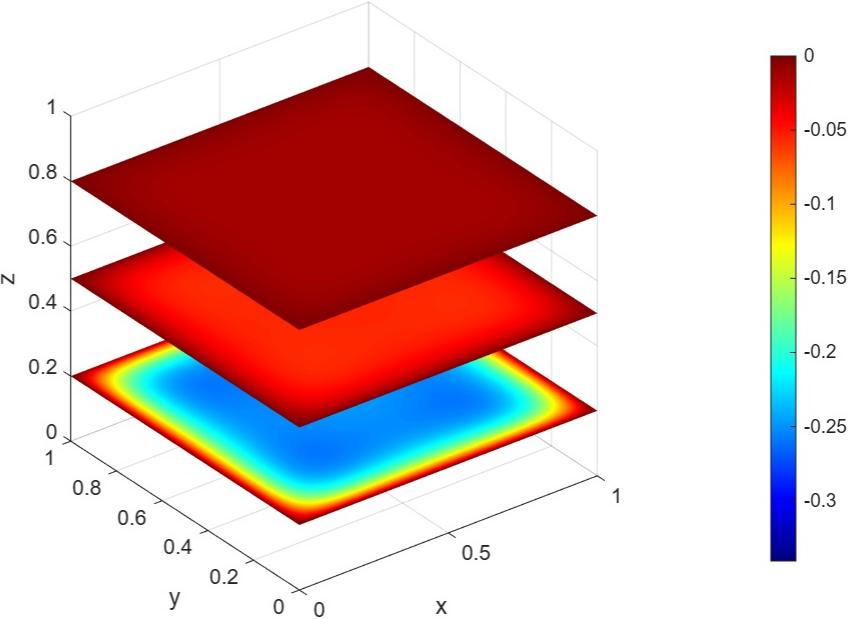}}
	\subfloat[Slices  at fixed 
	$y$-levels.]{\includegraphics[width=0.33\textwidth]{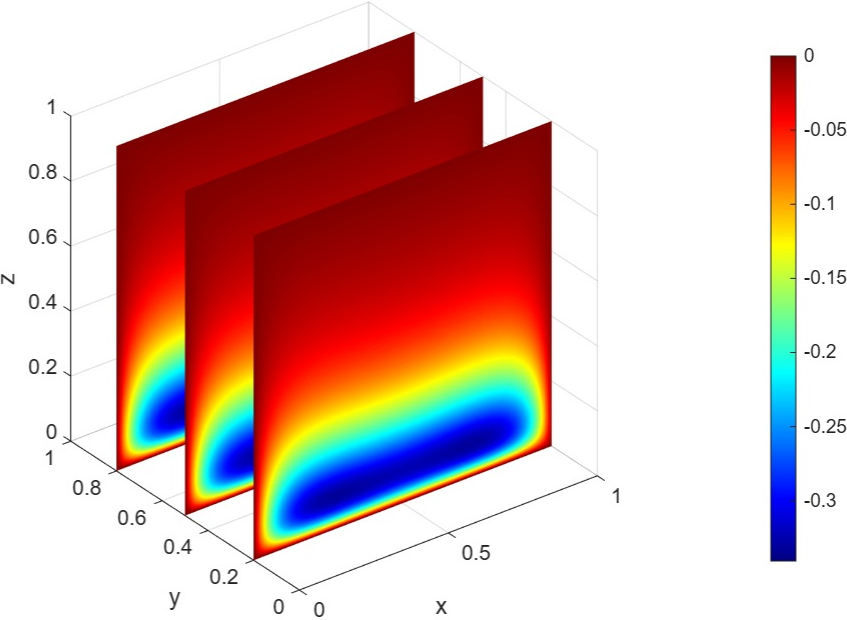}}
	\subfloat[Slices  at fixed $x$-levels.]{\includegraphics[width=0.33\textwidth]{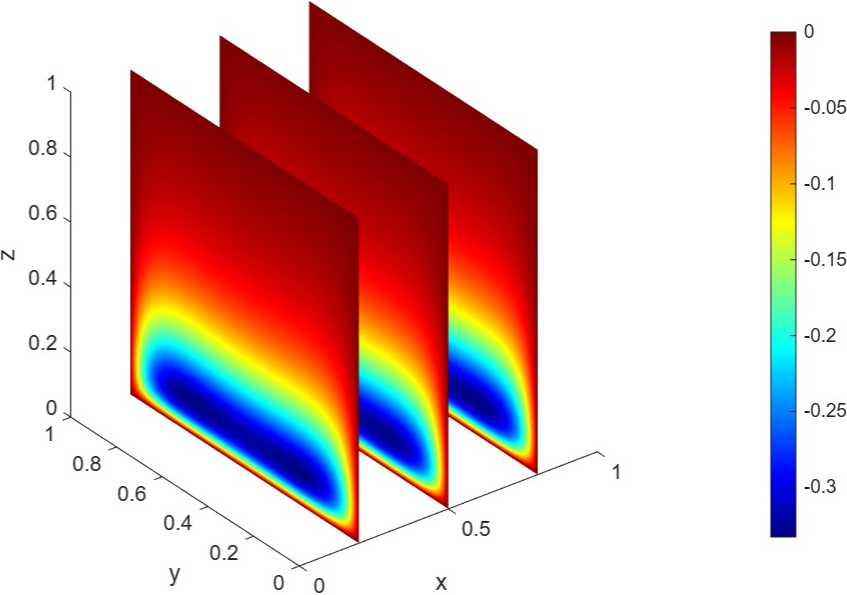}}
	\caption{Cross-sectional views of the numerical solution  ${\bf U}$ for Example~\ref{ex-3d-peak}, illustrating its boundary layer with $ (q,p)=(5,1)$, computed on a $256\times256\times256$ grid. }
	\label{fig-3d-peak-u2}
\end{figure}

\begin{example}\label{ex1_3d}
	We consider the following manufactured solution,
	\begin{equation*}
		u(x,y,z)=\sin(k_1x)\cos(k_2y)\sin(k_3z).
	\end{equation*}
	Such a solution is widely used for testing high-order schemes \cite{li2023high}. The parameters $k_1$, $k_2$, and $k_3$ control the oscillatory behavior of the solution. 
	% By choosing larger values of these parameters, one can effectively examine the performance of the numerical method for highly oscillatory problems.
\end{example}

Table~\ref{table-3d-sincos-var} shows that the proposed augmented HOC finite difference scheme performs robustly for the variable-coefficient problem under both parameter settings. 
For $(k_1,k_2,k_3)=(5,1,5)$, the error $\|{\bf E}^u\|_\infty$ exhibits essentially fourth-order convergence on all refined meshes. For the more oscillatory case $(k_1,k_2,k_3)=(25,5,25)$, the errors are much larger on the same mesh and the observed order is slightly lower on coarse grids, but the convergence rate still approaches fourth order as the mesh is refined. Meanwhile, the GMRES iteration counts remain very small in both cases, indicating that the augmented solve remains efficient even in the presence of variable coefficients.

\begin{table}[htbp]
	\footnotesize
	\centering
	\caption{Grid refinement analysis of the  augmented HOC finite difference scheme for Example~\ref{ex1_3d} with $\alpha = x^2+y^2+z^2+1$ and $\beta=x+y+z+10$.}
	\label{table-3d-sincos-var}
	\begin{tabular}{|c||c|c|c||c|c|c|}
		\hline
		& \multicolumn{3}{c||}{$(k_1,k_2,k_3)=(5,1,5)$} 
		& \multicolumn{3}{c|}{ $(k_1,k_2,k_3)=(25,5,25)$} \\
		\hline
		$N_x \times N_y \times N_z$ 
		& $||{\bf E}^u||_{\infty}$ & Order & Iters
		& $||{\bf E}^u||_{\infty}$ & Order & Iters
		\\
		\hline
		$32\times 32\times 32$
		&6.68E-06 & -- & 2&3.62E-03 & -- & 3\\
		$64\times 64\times 64$
		&4.20E-07 & 3.99 & 3 & 2.84E-04 & 3.68 & 4\\
		$128\times 128\times 128$
		&2.63E-08 & 4.00 &3 & 1.87E-05 & 3.92 & 4\\
		$256\times 256\times 256$
		&1.65E-09 & 4.00 & 4 & 1.19E-06 & 3.98 & 5\\
		$512\times 512\times 512$
		&1.05E-10 & 3.97 & 4 & 7.45E-08 &4.00 & 5\\
		
		\hline
	\end{tabular}
\end{table}

Figure~\ref{fig-3d-25525-u} displays cross-sectional views of the numerical solution ${\bf U}$ for Example~\ref{fig-3d-25525-u} with $(k_1,k_2,k_3)=(25,5,25)$ on a $256\times256\times256$ grid. The numerical solution shows a strongly oscillatory pattern, with alternating positive and negative regions distributed throughout the domain. The oscillations are noticeably stronger in the $x$- and $z$-directions, while the variation in the $y$-direction is relatively milder, reflecting the different frequencies prescribed by $(k_1,k_2,k_3)$. The three panels together illustrate the anisotropic oscillatory structure of the solution and confirm that the proposed method captures such highly oscillatory behavior sharply and consistently in three dimensions.

\begin{figure}[tbhp]
	\centering
	\subfloat[Slices  at fixed $z$-levels.]{\includegraphics[width=0.33\textwidth]{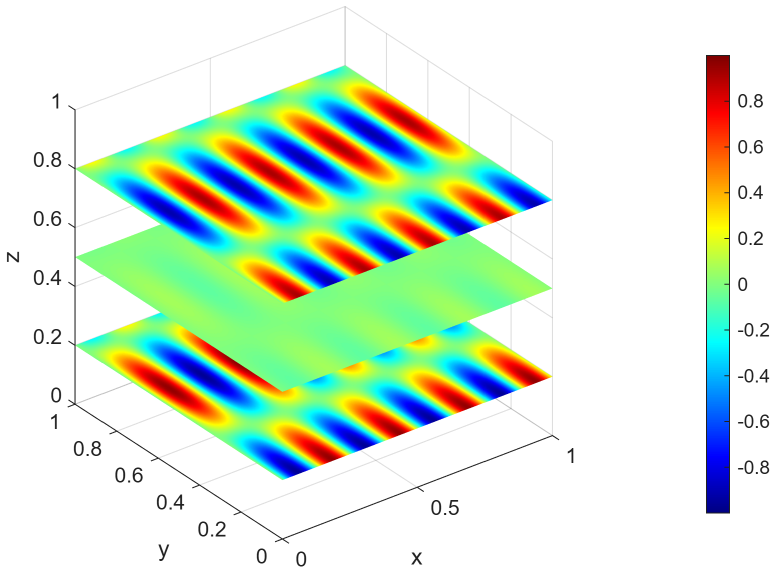}}
	\subfloat[Slices  at fixed 
	$y$-levels.]{\includegraphics[width=0.33\textwidth]{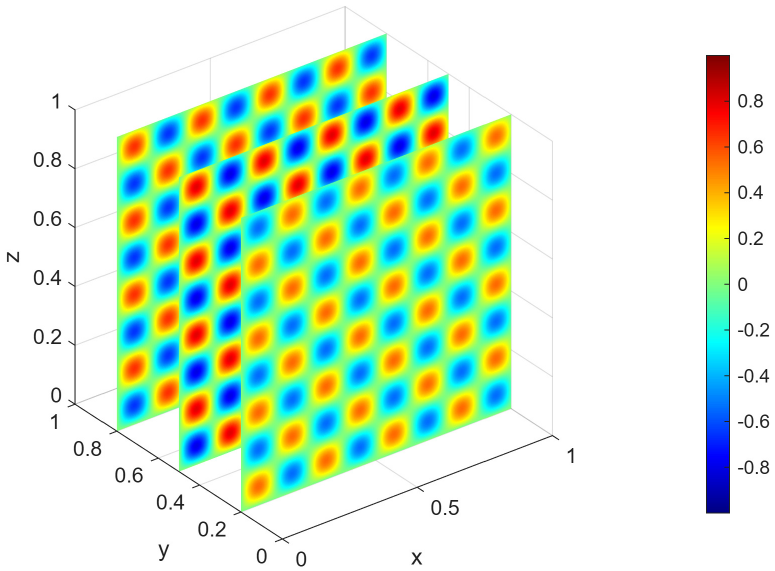}}
	\subfloat[Slices  at fixed $x$-levels.]{\includegraphics[width=0.33\textwidth]{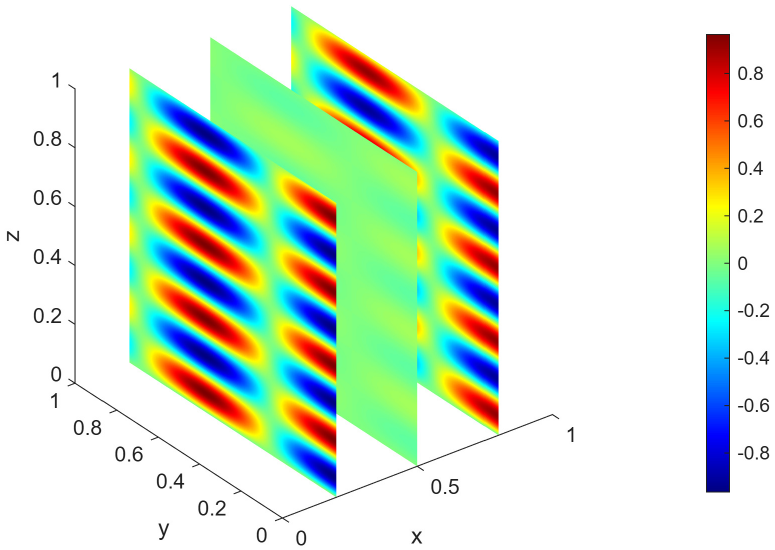}}
	\caption{Cross-sectional views of the numerical solution  ${\bf U}$ for Example~\ref{ex1_3d}, illustrating its oscillatory behavior with $ (k_1,k_2,k_3)=(25,5,25)$, computed on a $256\times256\times256$ grid. }
	\label{fig-3d-25525-u}
\end{figure}

Then, we examine the computational efficiency of the proposed method for the classical biharmonic equations $\Delta^2 u=f$ in terms of CPU time. We further employ OpenMP parallelization to accelerate the computation on a single node with multiple threads.
All tests were performed in Fortran 90 on a workstation with 4 Intel Xeon Gold 6248R processors (80 cores) and 1.5 TB of RAM.

Table~\ref{table-time} shows that the proposed method remains highly efficient even on very fine three-dimensional meshes.
In particular, on the finest $1024\times1024\times1024$ mesh, corresponding to a problem with over one billion grid points, the entire computation is completed in about $362.76$ seconds for Example~\ref{ex-3d-peak} and about $161.88$ seconds for Example~\ref{ex1_3d}. 
Meanwhile, the GMRES iteration counts for the augmented system remain very small on all meshes. 
This demonstrates that the augmented part of the algorithm introduces only a very mild additional cost and that the overall method is able to solve extremely large-scale problems with high efficiency.

\begin{table}[htbp]
	\footnotesize
	\centering
	\caption{CPU times for two three-dimensional examples obtained using 16 OpenMP threads.}
	\label{table-time}
	\begin{tabular}{|c||c|c||c|c|}
		\hline
		\multirow{2}{*}{$N_x\times N_y\times N_z$}
		& \multicolumn{2}{c||}{Example~\ref{ex-3d-peak} $(q,p)=(5,1)$}
		& \multicolumn{2}{c|}{Example~\ref{ex1_3d} $(k_1,k_2,k_3)=(25,5,25)$} \\
		\cline{2-5}
		& Iters & Time
		& Iters & Time \\
		\hline
		$128\times 128\times 128$
		& 5 & 0.78s
		& 6 & 0.28s \\
		
		$256\times 256\times 256$
		& 5 & 5.78s
		& 7 & 2.30s \\
		
		$512\times 512\times 512$
		& 5 & 44.86s
		& 8 & 19.89s \\
		
		$1024\times 1024\times 1024$
		& 5 & 362.76s
		& 8 & 161.88s \\
		\hline
	\end{tabular}
\end{table}

\subsection{Applications and benchmark problems}
The preceding sections developed and analyzed the augmented HOC method for the biharmonic equation. We now consider several related problems to illustrate the use of the same augmented strategy beyond the model problem.
These examples focus on the formulation and numerical performance of the resulting schemes, while a detailed analysis of each extension is left for future work.

\subsubsection{Triharmonic equation}

We next extend the proposed scheme to the triharmonic equation, namely the sixth-order elliptic problem
\begin{align}
	&\Delta^3 u = f, \qquad \text{in } \Omega,\\
	&u = u_n = u_{nn} = 0, \qquad \text{on } \partial\Omega_1,\\
	&u = \Delta u = \Delta^2 u = 0, \qquad \text{on } \partial\Omega_2,
\end{align}
where $\partial\Omega = \partial\Omega_1 \cup \partial\Omega_2$.

We briefly describe the corresponding solution strategy. By introducing auxiliary variables, the triharmonic equation can be reformulated as the mixed system
\begin{align*}
	\Delta u = v, \qquad \Delta v = w, \qquad \Delta w = f, \qquad \text{in } \Omega.
\end{align*}
The boundary conditions for $u$ and $v$ can be determined explicitly from the original problem. However, the boundary condition for $w$ is not available explicitly. To overcome this difficulty, we introduce an augmented variable $q$ on the boundary by setting $q = w$, which also satisfies the augmented equation $ q= \Delta^2 u$.
Once a high-order accurate approximation of $q$ is obtained, the Poisson equations for $w$, $v$, and $u$ can be solved sequentially by using the fourth-order compact finite difference scheme together with the FFT solver.

Here, we again use the same exact solution as in Example~\ref{ex1_3d}, and the corresponding right-hand side $f$ can be computed directly.
Table~\ref{table-3d-tri} shows that the proposed method performs robustly for the triharmonic equation under both parameter settings. In the case $(k_1,k_2,k_3)=(5,1,5)$, the error $\|{\bf E}^u\|_\infty$ exhibits essentially fourth-order convergence. For the more oscillatory case $(k_1,k_2,k_3)=(25,5,25)$, the errors are much larger on the same mesh, but the convergence rate still improves steadily and approaches fourth order on fine meshes. 

\begin{table}[htbp]
	\footnotesize
	\centering
	\caption{Grid refinement analysis of the  augmented HOC finite difference scheme for the triharmonic equation with two different choices of $(k_1,k_2,k_3)$.}
	\label{table-3d-tri}
	\begin{tabular}{|c||c|c||c|c|}
		\hline
		\multirow{2}{*}{$N_x \times N_y \times N_z$}
		& \multicolumn{2}{c||}{$(k_1,k_2,k_3)=(5,1,5)$}
		& \multicolumn{2}{c|}{$(k_1,k_2,k_3)=(25,5,25)$} \\
		\cline{2-5}
		& $||{\bf E}^u||_{\infty}$ & Order
		& $||{\bf E}^u||_{\infty}$ & Order \\
		\hline
		$32\times 32\times 32$
		& 1.57E-06 & --   
		& 4.10E-03 & --    \\
		
		$64\times 64\times 64$
		& 9.86E-08 & 3.99 
		& 2.67E-04 & 3.94 \\
		
		$128\times 128\times 128$
		& 6.18E-09 & 4.00 
		& 1.70E-05 & 3.97  \\
		
		$256\times 256\times 256$
		& 3.86E-10 & 4.00 
		& 1.07E-06 & 3.99  \\
		
		$512\times 512\times 512$
		& 2.26E-11 & 4.09 
		& 6.71E-08 & 4.00  \\
		\hline
	\end{tabular}
\end{table}

\subsubsection{High-wavenumber problem}
We consider the following fourth-order problem,
\begin{equation}\label{eq-high-frequency}
	-\Delta^2 u-k^2\Delta u
	=
	\frac{10^5}{\pi\sigma^2}
	\exp\Big(
	-\frac{\|{\bf x}-0.5\|^2_2}{\sigma^2}
	\Big),
	\qquad {\bf x}\in[0,1]^d,
\end{equation}
subject to the homogeneous clamped boundary conditions.
Here, $\sigma>0$ controls the width of the localized Gaussian source, while
the parameter $k$ determines the oscillatory behavior of the solution.

In Fig.~\ref{fig_hwp_err}, the grid-refinement results ($\sigma=0.02$) show that the maximum-norm errors decrease at approximately fourth order for all tested wave numbers once the meshes are sufficiently fine. 
The reference solutions used to evaluate the errors were computed on an $8192^2$ grid in two dimensions and a $1024^3$ grid in three dimensions. 
Some pre-asymptotic behavior is observed for larger $k$ on coarse grids because the oscillations are not adequately resolved. 
The representative solutions in Fig.~\ref{fig_hwp} further illustrate the localized and highly oscillatory structures captured by the proposed method.

\begin{figure}[htbp]
	\centering
	\begin{minipage}[b]{0.48\textwidth}
		\centering
		\includegraphics[width=\linewidth]{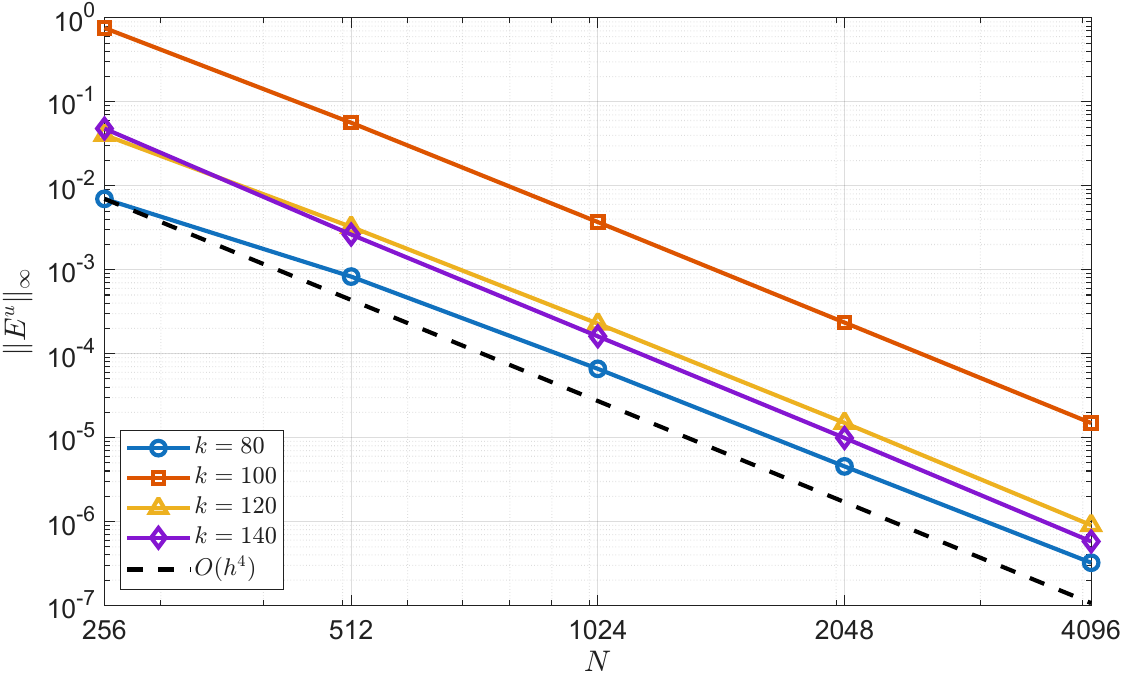}
		
		(a) $d=2$
	\end{minipage}
	\hfill
	\begin{minipage}[b]{0.48\textwidth}
		\centering
		\includegraphics[width=\linewidth]{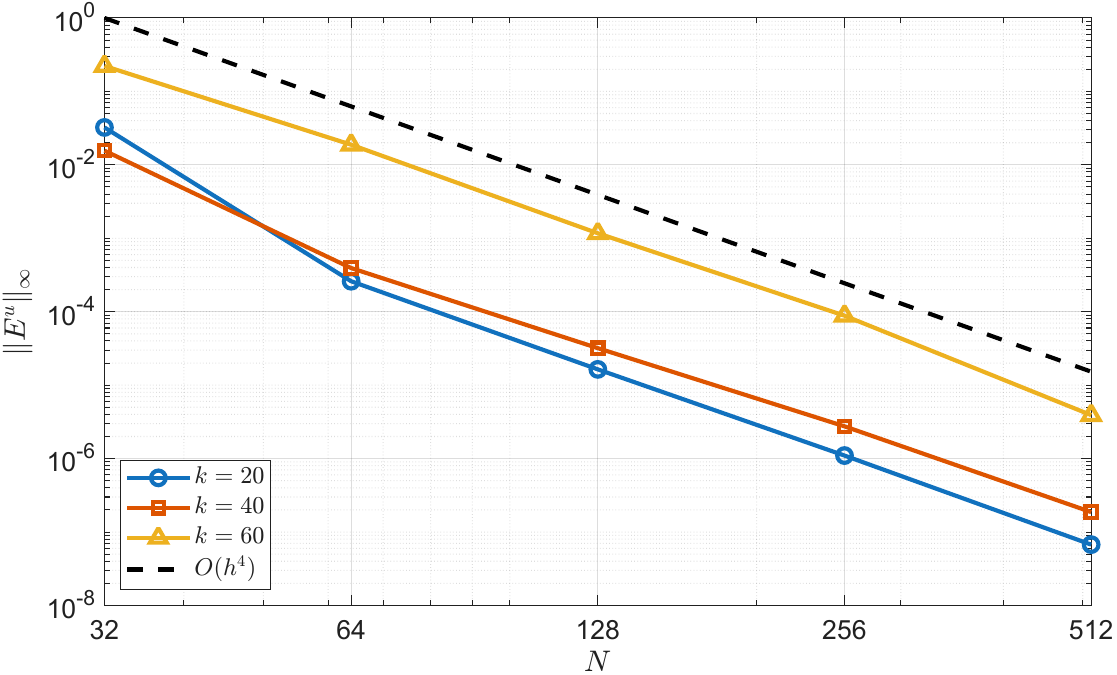}
		
		(b) $d=3$
	\end{minipage}
	\caption{Grid-refinement results for the maximum-norm error $\|{\bf E}^u\|_\infty$ with different wave numbers. }
	\label{fig_hwp_err}
\end{figure}

\begin{figure}[htbp]
	\centering
	\begin{minipage}[b]{0.48\textwidth}
		\centering
		\includegraphics[width=\linewidth]{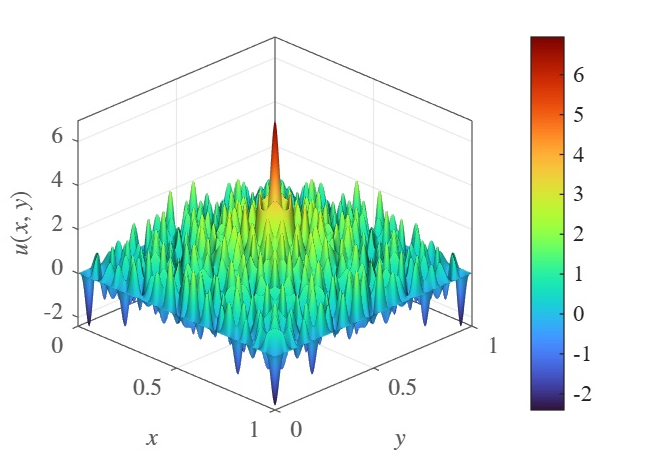}
		
		(a) Wave number $k=120$ for $d=2$
	\end{minipage}
	\hfill
	\begin{minipage}[b]{0.48\textwidth}
		\centering
		\includegraphics[width=\linewidth]{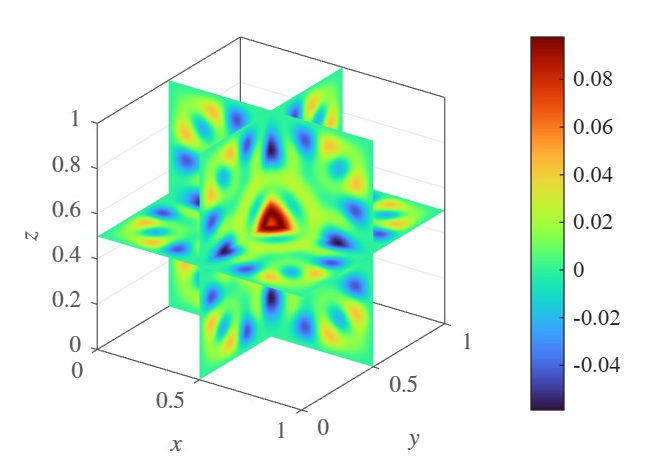}
		
		(b) Wave number $k=40$ for $d=3$
	\end{minipage}
	\caption{Numerical solutions for the two-dimensional problem with $k=120$ on a $4096^2$ grid and the three-dimensional problem with $k=40$ on a $64^3$ grid.}
	\label{fig_hwp}
\end{figure}

\subsubsection{Incompressible flows}

This example is the classical lid-driven cavity flow problem. The boundary conditions are prescribed by
\begin{equation}
	\begin{aligned}
		&u=0, \qquad (x,y)\in\partial\Omega,\\
		&[-u_y,u_x]=
		\begin{cases}
			[0,0], & x=0,1 \ \text{or}\ y=0,\\
			[-1,0], & y=1,
		\end{cases}
	\end{aligned}
\end{equation}
and the source term $f=0$.

Table~\ref{table-LDF} lists the maximum value of $|u|$ and its location for the lid-driven cavity flow problem on a sequence of refined meshes.
A comparison with the benchmark results available in the literature shows very good agreement in both the maximum value and its location. In particular, the present results are consistent with those reported by Bialecki~\cite{Bialecki2003}, Altas et al.~\cite{AtlasDymGuptaManohar1998}, Matania et al.~\cite{fishelov2012recent}, and Pan et al.~\cite{pan2025fourth}. This confirms the accuracy of the proposed method for the lid-driven cavity flow problem. Figure \ref{fig_ldf} presents the streamline plot obtained by the proposed mixed method.
We can clearly see  the center eddies and the corner eddies  at the bottom corners.
These results are consistent with those in the literatures (Figure 8 in \cite{Lin3735897};  Figure 3 in \cite{Shankar1744305}).

\begin{table}[htbp]
	\footnotesize
	\caption{Maximum value and location of $ |u|$ for the lid-driven cavity flow problem.}\label{table-LDF}
	\begin{center}
		\begin{tabular}{|c|l|c|} \hline
			$N$   & $\max |u|$  & location $(x,y)$  \\
			\hline
			1024 & 0.1000761035 & (0.5, 0.7646) \\
			2048 & 0.1000762526 & (0.5, 0.7651) \\
			4096 & 0.1000762526 & (0.5, 0.7651)\\
			Bialecki~\cite{Bialecki2003}($N =128$) &  0.100076276 &(0.5, 0.765)  \\
			Altas et al.~\cite{AtlasDymGuptaManohar1998}($N =64$) & 0.10008 & (0.5, 0.766) \\
			Ben-Artzi et al.~\cite{fishelov2012recent}($N=256$)&0.1000759 &(0.5, 0.765625) \\
			\hline
		\end{tabular}
	\end{center}
\end{table}

\begin{figure}[htbp]
	\centering
	\includegraphics[width=0.7\textwidth]{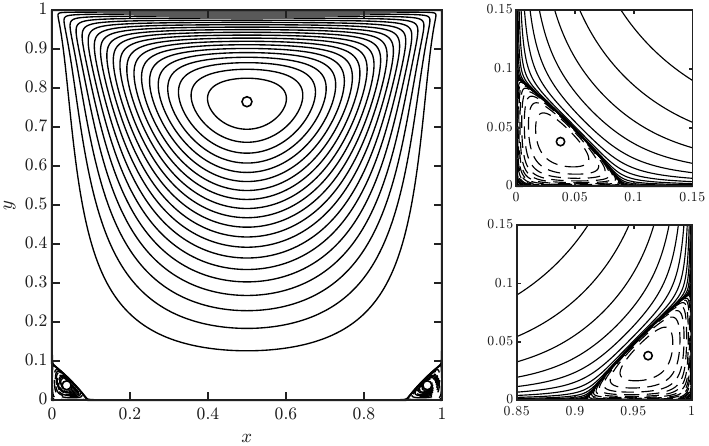}
	\caption{Streamline plot of lid-driven cavity flow.}
	\label{fig_ldf}
\end{figure}

\subsubsection{A clamped plate under a concentrated load}
This example models the bending of a clamped square plate under a unit load concentrated at the center. We take
$\Omega=(0,1)\times(0,1)$
with homogeneous clamped boundary conditions
\begin{equation}
	u=0,\qquad u_n=0,\qquad (x,y)\in\partial\Omega.
\end{equation}
The right-hand side is chosen as a mesh-dependent approximation of the unit point load,
\begin{equation}
	f(x,y)=
	\begin{cases}
		\dfrac{1}{4h^2}, & \left|x-\dfrac12\right|\le h \ \text{and}\ \left|y-\dfrac12\right|\le h,\\[2mm]
		0, & \text{otherwise}.
	\end{cases}
\end{equation}
The computed central deflection ${\bf U}(0.5,0.5)$ and bending moment $-{\bf V}(0.5,1)$ converge to $0.005612$ and $-0.12577$, respectively, in close agreement with the classical reference values $0.00560$ and $-0.1257$. This confirms the accuracy of the proposed method for the clamped plate problem under a unit concentrated load.

\section{Conclusions}
\label{sec:conclusions}

We have developed an augmented HOC scheme for
biharmonic equations with clamped boundary conditions in two and three
dimensions. By introducing the unknown boundary values of the auxiliary
variable as augmented variables, we reduce the original fourth-order problem
to decoupled second-order subproblems coupled through a lower-dimensional
Schur complement system. The subproblems are solved by FFT-based solvers,
and the Schur complement system is handled by matrix-free GMRES. High-order
convergence is proved using discrete stability estimates and discrete sine
expansions.

Numerical experiments confirm the accuracy, robustness, and efficiency of
the proposed method for both constant- and variable-coefficient problems.
Only a small number of GMRES iterations are required, and problems with more
than one billion unknowns are solved within a practical runtime. 
Furthermore, applications to triharmonic equations, incompressible flows, and plate bending demonstrate the applicability of the augmented framework to related high-order elliptic problems.
Future work will focus on extending the proposed approach to more general fourth-order PDEs and to problems posed on nonrectangular domains.

\appendix
\section{Proofs of Auxiliary Lemmas and Corollaries}
\begin{proof}[Proof of Lemma~\ref{lem-repre-ND} for $d=2$]
	For each fixed $i$, the homogeneous boundary conditions at $j=0,N$ imply
	the discrete sine expansion
	\[
	\psi_{i,j}
	=
	\sum_{k=1}^{N-1}c_k(i)\sin(\theta_k j),
	\qquad
	\theta_k=\frac{k\pi}{N}.
	\]
	Substitution into $L_h\psi_{i,j}=0$ and the linear independence of the
	discrete sine basis give
	\[
	(4+2\cos\theta_k)c_k(i-1)
	+
	(8\cos\theta_k-20)c_k(i)
	+
	(4+2\cos\theta_k)c_k(i+1)
	=0.
	\]
	Equivalently,
	\[
	c_k(i+1)-2\cosh(\alpha_k)c_k(i)+c_k(i-1)=0,
	\qquad
	\cosh\alpha_k
	=
	\frac{5-2\cos\theta_k}{2+\cos\theta_k}.
	\]
	The boundary conditions yield
	\[
	c_k(0)=\widehat g_k,\qquad c_k(N)=0,
	\qquad
	\widehat g_k
	=
	\frac{2}{N}\sum_{m=1}^{N-1}
	g_m\sin\frac{k\pi m}{N}.
	\]
	Solving the recurrence gives
	\[
	c_k(i)
	=
	\widehat g_k r_{k,i},
	\qquad
	r_{k,i}
	=
	\frac{\sinh((N-i)\alpha_k)}{\sinh(N\alpha_k)},
	\]
	which proves the representation \eqref{eq-varphi-repre-ND}.
	
	Finally, the discrete sine orthogonality relation
	\[
	h\sum_{j=1}^{N-1}
	\sin\frac{k\pi j}{N}
	\sin\frac{\ell\pi j}{N}
	=
	\frac{1}{2}\delta_{k\ell}
	\]
	implies
	\[
	\|\Psi\|_{2(\Omega_h^0)}^2
	=
	\frac{h}{2}
	\sum_{k=1}^{N-1}
	|\widehat g_k|^2
	\sum_{i=1}^{N-1}r_{k,i}^2.
	\]
\end{proof}

\begin{proof}[Proof of Lemma~\ref{lem-boundmw-ND} for $d=2$]
	Since $\alpha_k>0$,
	\[
	0\le r_{k,i}\le e^{-i\alpha_k},
	\]
	and hence
	\[
	h\sum_{i=1}^{N-1}r_{k,i}^2
	\le
	h\sum_{i=1}^{\infty}e^{-2i\alpha_k}
	\le
	\frac{h}{2\alpha_k}.
	\]
	Moreover,
	\[
	\cosh\alpha_k-1
	=
	\frac{3(1-\cos\theta_k)}{2+\cos\theta_k}
	\ge c\theta_k^2.
	\]
	If $\alpha_k\le1$, then
	$\cosh\alpha_k-1\le C\alpha_k^2$, whereas for $\alpha_k>1$ the same
	conclusion follows from $\theta_k\le\pi$. Therefore,
	\[
	\alpha_k\ge c\theta_k=ck\pi h.
	\]
	Consequently,
	\[
	h\sum_{i=1}^{N-1}r_{k,i}^2
	\le
	\frac{C}{k}
	\le C,
	\]
	which completes the proof.
\end{proof}

\begin{proof}[Proof of Corollary~\ref{cor-dpi-ND} for $d=2$]
	The sine expansion
	\[
	g_j
	=
	\sum_{k=1}^{N-1}
	\widehat g_k\sin\frac{k\pi j}{N}
	\]
	and discrete orthogonality give the Parseval identity
	\[
	h\sum_{j=1}^{N-1}|g_j|^2
	=
	\frac{1}{2}
	\sum_{k=1}^{N-1}|\widehat g_k|^2.
	\]
\end{proof}

\section*{Acknowledgments}
Kejia Pan was supported by Hunan Basic Science Research Center for Mathematical Analysis (2024JC2002). The work of the fourth author: L. Wang is partially supported by Singapore MOE Tier-2 project: MOE-T2EP20224-0012. 
The authors are grateful for resources from the High Performance Computing Center of Central South University.

\bibliographystyle{unsrtnat}
\bibliography{references}

\end{document}